\documentclass{amsart}
\usepackage{amssymb}
\usepackage{mathrsfs}
\usepackage{graphicx}
\usepackage{url}
\usepackage{color}

\newtheorem{theorem}{Theorem}[section]
\newtheorem{proposition}[theorem]{Proposition}
\newtheorem{corollary}[theorem]{Corollary}
\newtheorem{lemma}[theorem]{Lemma}
\theoremstyle{remark}
\newtheorem{definition}[theorem]{Definition}
\newtheorem{remark}[theorem]{Remark}
\newtheorem{example}[theorem]{Example}
\newtheorem{problem}[theorem]{Problem}

\numberwithin{equation}{section}

\begin{document}

\begin{abstract}
We investigate equiangular lines in finite orthogonal geometries, focusing specifically on equiangular tight frames (ETFs).
In parallel with the known correspondence between real ETFs and strongly regular graphs (SRGs) that satisfy certain parameter constraints, we prove that ETFs in finite orthogonal geometries are closely aligned with a modular generalization of SRGs.
The constraints in our finite field setting are weaker, and all but~18 known SRG parameters on $v \leq 1300$ vertices satisfy at least one of them.
Applying our results to triangular graphs, we deduce that Gerzon's bound is attained in finite orthogonal geometries of infinitely many dimensions.
We also demonstrate connections with real ETFs, and derive necessary conditions for ETFs in finite orthogonal geometries.
As an application, we show that Gerzon's bound cannot be attained in a finite orthogonal geometry of dimension~5.
\end{abstract}

\title[Frames over finite fields: Equiangular lines in orthogonal geometry]{Frames over finite fields: \\ Equiangular lines in orthogonal geometry}
\author{Gary R.~W.~Greaves}
\address{School of Physical and Mathematical Sciences, Nanyang Technological University, Singapore 637371}
\email{gary@ntu.edu.sg}
\author{Joseph W.~Iverson}
\address{Department of Mathematics, Iowa State University, Ames, IA USA 50011}
\email{jwi@iastate.edu}
\author{John Jasper}
\address{Department of Mathematics and Statistics, South Dakota State University, Brookings, SD USA 57007}
\email{john.jasper@sdstate.edu}
\author{Dustin G.~Mixon}
\address{Department of Mathematics, The Ohio State University, Columbus, OH USA 43210}
\email{mixon.23@osu.edu}


\maketitle

\section{Introduction}

Over the real number field, a sequence $\mathscr{L} = \{ \ell_j \}_{j \in [n]}$ of lines through the origin of $\mathbb{R}^d$ is called \textbf{equiangular} if there are constants $a,b \geq 0$ and nonzero column vectors $\varphi_j \in \ell_j$ such that $\varphi_j^\top \varphi_j = a$ and $(\varphi_i^\top \varphi_j)^2 = b$ for every $i \neq j$ in $[n]:= \{ 1,\dotsc,n\}$.
As a special instance, the representatives $\{ \varphi_j \}_{j\in [n]}$ are said to form an \textbf{equiangular tight frame} (ETF) of size $d\times n$ if $\mathscr{L}$ is equiangular and the matrix $\Phi = \left[ \begin{array}{ccc} \varphi_1 & \cdots & \varphi_n \end{array} \right] \in \mathbb{R}^{d\times n}$ satisfies $\Phi \Phi^\top = cI$ for some~$c$.
ETFs are known to create optimal packings in projective space $\mathbb{R}\mathrm{P}^{d-1}$, and for this reason and others they find applications in areas such as compressed sensing~\cite{BFMW13}, wireless communication~\cite{SH03}, digital fingerprinting~\cite{MQKF13}, and quantum information theory~\cite{RBSC04}.

In this paper, we investigate the related phenomenon of ETFs over finite fields (Definition~\ref{def: ETF}).
Frames and equiangular lines over finite fields are introduced in the companion paper~\cite{FFF1}, which develops the basic theory and continues on to study ETFs in unitary geometries (akin to complex ETFs).
Here we focus on ETFs in orthogonal geometries, which are the finite field analogs of real ETFs.

In the real setting, ETFs are deeply connected with the theory of strongly regular graphs (SRGs).
Using a construction of Seidel, an SRG with parameters $(v,k,\lambda,\mu)$ creates a real ETF with $n = v+1$ vectors whenever $k = 2 \mu$, and, up to a change of line representatives, every ETF of $n > d + 1$ vectors in $\mathbb{R}^d$ arises this way~\cite{S91,T77,W09}.
In another construction, an SRG with parameters $(v,k,\lambda,\mu)$ creates a real ETF with $n = v$ vectors whenever $v = 4k - 2\lambda - 2\mu$, and this correspondence uniquely characterizes real ETFs with \emph{centroidal symmetry}~\cite{T77,FJMPW18}.
From this perspective, real ETFs may be viewed as a subclass of SRGs, where about 27\% of known SRG parameters on $v \leq 1300$ vertices satisfy one of the constraints $k = 2\mu$ or $v = 4k - 2\lambda - 2\mu$~\cite{B}.
This relationship has proven to be fruitful for the theories of both SRGs and ETFs, and results from each area have been used to establish both existence and nonexistence in the other~\cite{B,FM:T,Mak02,STDH07,FJMP18,AM18,AM20}.

A similar phenomenon occurs for ETFs in finite orthogonal geometries, but in this setting we encounter the more general notion of a \textbf{modular strongly regular graph} (Definition~\ref{def: mod SRG}).
As with SRGs, modular SRGs are described by parameters $(v,k,\lambda,\mu)$, and an ETF can be obtained if and only if constraints like those above are satisfied (Theorem~\ref{thm: Seidel Waldron} and Theorem~\ref{thm: centered or axial}).
However, in the finite field setting these parameter constraints need only be satisfied modulo the field characteristic.
Thanks to this weaker condition, a much larger fraction of SRGs can be identified with ETFs over finite fields, and we observe that 98\% of known SRG parameters on $v \leq 1300$ vertices occur in this way (Remark~\ref{rem: SRGs from ETFs}).

These constructions produce remarkable examples of ETFs over finite fields.
In the real setting, \textbf{Gerzon's bound} states that $\mathbb{R}^d$ admits a system of $n$ equiangular lines only if $n \leq \binom{d+1}{2}$.
Saturation is empirically rare, and examples of $n = \binom{d+1}{2}$ equiangular lines in $\mathbb{R}^d$ are known only for $d \in \{2,3,7,23\}$.
It is an open problem whether the bound can be attained in any other dimension, and it is known that this can happen only for $d \geq 119$ such that $d+2$ is an odd square~\cite{LS73,Mak02,BMV04}.
Gerzon's bound also applies in our finite field setting (Proposition~\ref{prop: Gerzon}).
However, we now find a completely different behavior: when applied to triangular graphs, our results show that Gerzon's bound is saturated in some $d$-dimensional finite orthogonal geometry whenever $|d-7|$ is not a power of~2 (Theorem~\ref{thm: triangular gerzon}).

This is all the more surprising since we show that ETFs in finite orthogonal geometries are closely aligned with real ETFs.
If a $d\times n$ real ETF exists, then $d\times n$ ETFs exist in finite orthogonal geometries of all but finitely many characteristics (Proposition~\ref{prop: project real ETF}).
Conversely, if a $d\times n$ ETF exists in a finite orthogonal geometry of characteristic $p > 2n -5$, then a real $d\times n$ ETF also exists (Proposition~\ref{prop: no wrap implies real}).
In this sense, ETFs in finite orthogonal geometries may be seen as approximations of real ETFs, and there is only a narrow window in which existence may differ between the two settings.
Nevertheless, marked deviations occur within that window, and our results imply the existence of a large quantity of ETFs having sizes that are not known to exist over either the real or complex numbers (Example~\ref{ex: green SRGs}).

The paper is laid out as follows.
The next section provides background on SRGs, orthogonal geometry, and frames and equiangular lines over finite fields.
In Section~\ref{sec: fields}, we examine interactions between real ETFs and those in finite orthogonal geometries, and provide additional necessary conditions for the existence of the latter.
Section~\ref{sec: SRG} develops the crucial relationship between ETFs in finite orthogonal geometries and modular SRGs.
In Section~\ref{sec: ex} we give consequences of that theory and find many new ETF sizes over finite fields, including those that attain Gerzon's bound in infinitely many dimensions.
Finally, in Section~\ref{sec: 5x15} we give a computer-assisted proof that Gerzon's bound is never attained in a 5-dimensional finite orthogonal geometry of odd characteristic.

\section{Preliminaries}
\label{sec: prelim}

\subsection{Review of strongly regular graphs}
The reader may consult~\cite{BVM} for background on strongly regular graphs.
In this paper, all graphs are assumed to be simple and undirected.
If $\Gamma$ is a graph on the $v$-element vertex set $V$, then its $\{0,1\}$-adjacency matrix is defined as the $V \times V$ matrix $A$ with $A_{ij} = 1$ if vertices $i$ and $j$ are adjacent, and $A_{ij} = 0$ otherwise, whereas its \textbf{Seidel adjacency matrix} is the $V \times V$ matrix $\Sigma$ with entries
\[
\Sigma_{ij} = \begin{cases}
-1, & \text{if $i$ and $j$ are adjacent}; \\
0, & \text{if } i = j; \\
1, & \text{otherwise.}
\end{cases}
\]
We denote $J_v$ or simply $J$ for the $v \times v$ matrix of all ones, and $\mathbf{1}_v$ or $\mathbf{1}$ for the column vector of all ones.
Then the two adjacency matrices are related by
\begin{equation}
\label{eq: Seidel from 01}
\Sigma = J -2A - I.
\end{equation}

By definition, a nontrivial \textbf{strongly regular graph} (SRG) $\Gamma$ is neither complete nor edgeless, and there are parameters $v,k,\lambda,\mu$ such that its $\{0,1\}$-adjacency matrix $A$ has size $v \times v$ and satisfies 
\begin{equation}
\label{eq: SRG A}
AJ = kJ, \qquad A^2 = \mu J + (\lambda - \mu)A + (k-\mu) I.
\end{equation}
We abbreviate this by saying $\Gamma$ is a $(v,k,\lambda,\mu)$-SRG.
It has exactly two eigenvalues $r > s$ with eigenvectors outside the span of $\mathbf{1}$, namely
\[
r = \frac{1}{2}\left[ \lambda - \mu + \sqrt{ (\lambda - \mu)^2 + 4(k - \mu) } \right],
\quad
s = \frac{1}{2}\left[ \lambda - \mu - \sqrt{ (\lambda - \mu)^2 + 4(k - \mu) } \right].
\]
Each satisfies the equation 
\begin{equation}
\label{eq: eig quad}
x^2 = (\lambda - \mu)x + (k-\mu) \qquad \text{for }x \in \{r,s\},
\end{equation}
as well as the relations
\begin{equation}
\label{eq: SRG eig rels}
r+s = \lambda - \mu, \qquad rs = \mu - k, \qquad (k-r)(k-s) = \mu v.
\end{equation}
The orthogonal projections
\[
E_r = \tfrac{1}{r-s}( A - sI - \tfrac{k-s}{v} J ),
\qquad
E_s = \tfrac{1}{s-r}( A - rI - \tfrac{k-r}{v} J )
\]
together with $\tfrac{1}{v}J$ form a basis of mutually orthogonal idempotents for the algebra spanned by $A,I,J$.
Here $I = \tfrac{1}{v}J + E_r + E_s$ and $A = k(\tfrac{1}{v}J) + r E_r + s E_s$.
We denote $f = \operatorname{rank} E_r = (r-s)^{-1}(s-k-vs)$ and $g = \operatorname{rank} E_s = (s-r)^{-1}(r - k - vr )$.
If $f \neq g$ then $r$ and $s$ are integers, and if $f = g$ then $(v,k,\lambda,\mu) = (4\mu+1,2\mu,\mu-1,\mu)$.
The complementary graph $\overline{\Gamma}$ is also an SRG, with parameters $v' = v$, $k' = v-k-1$, $\lambda' = v-2k+\mu-2$, $\mu' = v-2k+\lambda$ and $r' = -s-1$, $s' = -r-1$, $f' = g$, $g' = f$.

We denote the eigenvalues of the Seidel adjacency matrix $\Sigma = J - 2A - I$ of $\Gamma$ by
\begin{equation}
\label{eq: SRG sig eigs}
\theta_k = v - 2k - 1, \qquad \theta_r = -2r - 1, \qquad \theta_s = -2s - 1.
\end{equation}
Here $\Sigma = \theta_k(\tfrac{1}{v} J) + \theta_r E_r + \theta_s E_s$, and in particular $\Sigma \mathbf{1} = \theta_k \mathbf{1}$.
From~\eqref{eq: SRG eig rels} we obtain
\begin{equation}
\label{eq: SRG eig rels 2}
\tfrac{1}{v}(\theta_k - \theta_r)(\theta_k - \theta_s) = v - 4k + 2\lambda + 2\mu,
\end{equation}
and the spectral decomposition of $\Sigma$ implies
\begin{equation}
\label{eq: SRG S1}
(\Sigma - \theta_r I)(\Sigma - \theta_s I) 
= (\theta_k - \theta_r)(\theta_k - \theta_s)(\tfrac{1}{v}J)
= (v - 4k + 2\lambda + 2\mu) J.
\end{equation}

\subsection{Review of orthogonal geometry}
The reader may consult~\cite{Gr02} for background on orthogonal geometry.

\begin{definition}
\label{def: orth geom}
Given an odd prime power $q$, a \textbf{quadratic space} over $\mathbb{F}_q$ is a finite-dimensional vector space $W$ equipped with a form $\langle \cdot, \cdot \rangle \colon W \times W \to \mathbb{F}_q$ that satisfies the following:
	\begin{itemize}
	\item[(O1)]
	for every $u \in W$, the induced mapping $\langle u, \cdot \rangle \colon W \to \mathbb{F}_q$ is linear,
	\smallskip
	
	\item[(O2)]
	$\langle u, w \rangle = \langle w, u \rangle$ for every $u,w \in W$,
	\smallskip
	
	\item[(O3)]
	if $u \in W$ satisfies $\langle u, w \rangle = 0$ for every $w \in W$, then $u = 0$.
	\end{itemize}
In other words, $\langle \cdot, \cdot \rangle$ is a nondegenerate symmetric bilinear form on $W$.
We also say that $\langle \cdot, \cdot \rangle$ induces an \textbf{orthogonal geometry} on $W$.
\end{definition}

In Definition~\ref{def: orth geom}, if $\{ e_j \}_{j \in [d]}$ is a basis for $W$, then the form $\langle \cdot, \cdot \rangle$ above is completely determined by its \textbf{Gram matrix} $M = \begin{bmatrix} \langle e_i, e_j \rangle \end{bmatrix}_{i,j \in [d]}$.
It is symmetric and invertible, and conversely any symmetric invertible $d\times d$ matrix over $\mathbb{F}_q$ is the Gram matrix of an orthogonal geometry on $W$.
The determinant of the Gram matrix $M$ lies in the multiplicative group $\mathbb{F}_q^\times$ of nonzero elements in $\mathbb{F}_q$, and if we denote
\[ \mathbb{F}_q^{\times 2} = \{ \alpha^2 : \alpha \in \mathbb{F}_q^\times \} \leq \mathbb{F}_q^\times \]
for the subgroup of nonzero quadratic residues, then the quadratic space $W$ is said to have \textbf{discriminant}
\[
\operatorname{discr}(W) = (\det M)\, \mathbb{F}_q^{\times 2} \in \mathbb{F}_q^\times / \mathbb{F}_q^{\times 2}.
\]
The latter does not depend on the choice of basis for $W$.

\begin{definition}
For an odd prime power $q$, the \textbf{real model} on $\mathbb{F}_q^d$ is given by the nondegenerate symmetric bilinear form $(x,y) = x^\top y$.
\end{definition}

Up to isometric isomorphism, a given vector space $W$ has exactly two types of orthogonal geometry: one for which the discriminant is trivial (square-type), and another for which it is not.
For example, the real model on the vector space $W = \mathbb{F}_q^d$ has Gram matrix $M = I$ with $\det M = 1$, so it is an orthogonal geometry with square-type discriminant.
As an example of the other type of orthogonal geometry on $\mathbb{F}_q^d$, let $\zeta \in \mathbb{F}_q^\times$ be a primitive element and create $\langle \cdot, \cdot \rangle$ from the Gram matrix $M = \operatorname{diag}(1,\dotsc,1,\zeta)$.
Then $\det M = \zeta \notin \mathbb{F}_q^{\times 2}$, so that the geometry induced by $\langle \cdot, \cdot \rangle$ on $\mathbb{F}_q^d$ has nontrivial discriminant.

If $W$ is a quadratic space, then a subspace $U \leq W$ is called \textbf{totally isotropic} if $\langle u, w \rangle = 0$ for every $u,w \in U$.
In that case $\dim U \leq \tfrac{1}{2} \dim W$.

\subsection{Review of frames and equiangular lines}
The basic theory of frames and equiangular lines over finite fields is developed in the companion paper~\cite{FFF1}.
We summarize the relevant portions here.

Throughout the paper, we abuse notation by identifying a sequence $\Phi = \{ \varphi_j \}_{j\in [n]}$ of vectors in a quadratic space $W$ with its \textbf{synthesis operator} $\Phi \colon \mathbb{F}_q^n \to W$ given by
\[
\Phi \{ x_i \}_{i\in [n]} = \sum_{i \in [n]} x_i \varphi_i.
\]
Its \textbf{analysis operator} $\Phi^\dagger \colon W \to \mathbb{F}_q^n$ is given by $\Phi^\dagger w = \begin{bmatrix} \langle \varphi_i, w \rangle \end{bmatrix}_{i\in [n]}$, and its \textbf{Gram matrix} is the matrix for $\Phi^\dagger \Phi$ in the standard basis, namely $G = \begin{bmatrix} \langle \varphi_i, \varphi_j \rangle \end{bmatrix} \in \mathbb{F}_q^{n\times n}$.

\begin{definition}
Let $W$ be a $d$-dimensional quadratic space over $\mathbb{F}_q$, where $q$ is odd.
A sequence $\Phi = \{ \varphi_j \}_{j\in [n]}$ in $W$ is called a \textbf{frame} of size $d\times n$ if it spans $W$.
A \textbf{tight frame} is a frame $\Phi$ for which there exists a \textbf{frame constant} $c \in \mathbb{F}_q$ such that $\Phi \Phi^\dagger = c I$.
We also call $\Phi$ a \textbf{$c$-tight frame} in that case.
If $\Phi = \{ \varphi_j \}_{j\in [n]}$ is a $c$-tight frame and there exists $a \in \mathbb{F}_q$ such that $\langle \varphi_j, \varphi_j \rangle = a$ for every $j \in [n]$, then we call $\Phi$ an \textbf{equal norm tight frame}, or an \textbf{$(a,c)$-NTF}.
\end{definition}

Unlike the real setting, it may happen that the frame constant equals zero.
We emphasize that the vectors of a $0$-tight frame span the ambient space by definition.
If $\Phi$ is an $(a,c)$-NTF of size $d\times n$, then the equal traces of $\Phi^\dagger \Phi$ and $\Phi \Phi^\dagger$ give
\begin{equation}
\label{eq: nadc}
na = dc.
\end{equation}

\begin{proposition}[\cite{FFF1}]
\label{prop: iso frames}
If $\Phi$ is a $d \times n$ frame, then it is a $0$-tight frame if and only if $\operatorname{im} \Phi^\dagger$ is a totally isotropic subspace of $\mathbb{F}_q^n$ in the real model.
In particular, a $0$-tight frame of size $d \times n$ exists only if $n \geq 2d$.
\end{proposition}

As in the real setting, it is possible to infer the existence of a frame from a corresponding Gram matrix.
In addition, the Gram matrix uniquely determines the discriminant of the ambient orthogonal geometry, which may be detected through the following notion.

\begin{definition}
\label{def: basic submatrix}
Given an $n \times n$ matrix $A = \left[ \begin{array}{ccc} a_1 & \cdots & a_n \end{array} \right] = \begin{bmatrix} A_{ij} \end{bmatrix}_{i,j \in [n]}$, choose any indexing set $J \subseteq [n]$ such that $\{ a_j \}_{j\in J}$ forms a basis for $\operatorname{im} A$.
Then the principal submatrix $A_b := \begin{bmatrix} A_{ij} \end{bmatrix}_{i,j \in J}$ is called a \textbf{basic submatrix} of $A$.
\end{definition}

In the following characterization, notice the absence of any condition akin to positive semidefiniteness.

\begin{proposition}[\cite{FFF1}]
\label{prop: factor Gram}
Let $G \in \mathbb{F}_q^{n\times n}$ ($q$ odd) be a square matrix, and let $G_b$ be a basic submatrix (Definition~\ref{def: basic submatrix}).
Given a quadratic space $W$, $G$ is the Gram matrix of a frame for $W$ if and only if all of the following are satisfied:
	\begin{itemize}
	\item[(i)]
	$G^\top = G$,
	\item[(ii)]
	$\operatorname{rank} G = \dim W$,
	\item[(iii)]
	$(\det G_b) \mathbb{F}_q^{\times 2} = \operatorname{discr}(W)$.
	\end{itemize}
In particular, $G$ is symmetric if and only if it is the Gram matrix of a frame for an orthogonal geometry with discriminant $(\det G_b) \mathbb{F}_q^{\times 2}$ on $\mathbb{F}_q^{\operatorname{rank} G}$.
\end{proposition}

\begin{proposition}[\cite{FFF1}]
\label{prop: tight Gram}
Suppose $G \in \mathbb{F}_q^{n\times n}$ is the Gram matrix of a frame $\Phi$.
Then, $G^2 = cG$ if and only if $\Phi$ is a $c$-tight frame.
\end{proposition}

\begin{definition}
\label{def: ETF}
Let $\Phi = \{ \varphi_j \}_{j \in [n]}$ be a sequence of vectors in a quadratic space $W$ over $\mathbb{F}_q$.
For $a,b \in \mathbb{F}_q$, $\Phi$ is called an \textbf{$(a,b)$-equiangular system} if:
	\begin{itemize}
	\item[(i)]
	$\langle \varphi_j, \varphi_j \rangle = a$ for every $j \in [n]$,
	\item[(ii)]
	$\langle \varphi_i, \varphi_j \rangle^2 = b$ for every $i \neq j$ in $[n]$.
	\end{itemize}
If this occurs and $\varphi_j \neq 0$ for every $j$, then we call $\mathscr{L} = \{ \operatorname{span} \varphi_j \}_{j\in [n]}$ a sequence of \textbf{equiangular lines}.
An $(a,b)$-equiangular system that is also a $c$-tight frame is known as an \textbf{equiangular tight frame}, or \textbf{$(a,b,c)$-ETF}.
\end{definition}

We have the following version of Gerzon's bound~\cite{LS73}.

\begin{proposition}[Gerzon's bound, \cite{FFF1}]
\label{prop: Gerzon}
Suppose $q$ is an odd prime power and $\Phi$ is an $(a,b)$-equiangular system of $n$ vectors in a $d$-dimensional quadratic space over $\mathbb{F}_q$, such that $a^2 \neq b$.
Then $n \leq \binom{d+1}{2}$,  and equality holds only if $\Phi$ is an ETF.
\end{proposition}

The parameters of an $(a,b,c)$-ETF are related in a manner similar to that given by the Welch bound in the real and complex settings~\cite{W74}.

\begin{proposition}[\cite{FFF1}]
\label{prop: Welch}
If $\Phi$ is an $(a,b,c)$-ETF of size $d\times n$, then 
\[ a(c-a) = (n-1)b. \]
\end{proposition}

Finally, the sizes of ETFs with nonzero frame constant come in Naimark complementary pairs.
We emphasize that the two geometries mentioned below may have different discriminants.

\begin{proposition}[Naimark complements, \cite{FFF1}]
\label{prop: Naimark}
If $c \neq 0$ and there is an $(a,b,c)$-ETF in an orthogonal geometry on $\mathbb{F}_q^d$ ($q$ odd), then there is a $(c-a,b,c)$-ETF in an orthogonal geometry on $\mathbb{F}_q^{n-d}$.
\end{proposition}

\subsection{Initial observations}
Before beginning in earnest, we make a few basic observations about ETFs in orthogonal geometries.
First of all, the existence of an ETF may be detected from its Gram matrix, where the following is immediate from Proposition~\ref{prop: factor Gram} and Proposition~\ref{prop: tight Gram}.

\begin{proposition}
\label{prop: orthogonal ETF from Gram matrix}
Suppose $G$ is a symmetric $n\times n$ matrix with entries in $\mathbb{F}_q$, where $q$ is odd.
Given $a,b,c \in \mathbb{F}_q$, $G$ is the Gram matrix of an $(a,b,c)$-ETF in an orthogonal geometry on $\mathbb{F}_q^d$ if and only if all of the following hold:
	\begin{itemize}
	\item[(ii)]
	$G_{ii} = a$ for every $i \in [n]$,
	\item[(ii)]
	$(G_{ij})^2 = b$ for every $i \neq j$ in $[n]$,
	\item[(iii)]
	$G^2 = c G$,
	\item[(iv)]
	$\operatorname{rank} G = d$.
	\end{itemize}
\end{proposition}

\begin{remark}
Proposition~\ref{prop: orthogonal ETF from Gram matrix} fails to mention the type of orthogonal geometry on $\mathbb{F}_q^d$, even though it is uniquely determined by $G$ and can be recovered from the determinant of a basic submatrix.
This is easy to compute in practice but may be difficult to predict in general.
For this reason, we shall often be deliberately ambiguous as we have done in Proposition~\ref{prop: orthogonal ETF from Gram matrix}, referring only to ``an'' orthogonal geometry without specifying its discriminant.

In any case, the discriminant can always be made unambiguous by passing to a quadratic extension field: any symmetric $G \in \mathbb{F}_q^{n\times n}$ is the Gram matrix of a frame for $\mathbb{F}_{q^2}^{\operatorname{rank} G}$ in the real model, since the determinant of a basic submatrix lies in $\mathbb{F}_q^\times \leq \mathbb{F}_{q^2}^{\times 2}$.
In particular, if an $(a,b,c)$-ETF exists in an orthogonal geometry on $\mathbb{F}_q^d$, then an $(a,b,c)$-ETF exists in the real model on $\mathbb{F}_{q^2}^d$.
\end{remark}

\begin{remark}
We will see later that, as far as existence of $d\times n$ ETFs is concerned, one may assume the field order $q = p^l$ is either prime (when $n \neq 2d$) or prime squared (when $n = 2d$).
As in the previous remark, one may further assume the geometry is the real model by taking $q = p^2$ (when $n \neq 2d$) or $q = p^4$ (when $n = 2d$).
See Theorem~\ref{thm: orth ETF in Fp}.
\end{remark}

\begin{remark}
\label{rem: rescale Gram}
The parameters of an $(a,b,c)$-ETF may be altered by rescaling its Gram matrix.
If $G \in \mathbb{F}_q^{n\times n}$ is the Gram matrix of an $(a,b,c)$-ETF in an orthogonal geometry and $\alpha \in \mathbb{F}_q^\times$, then $\alpha G$ is the Gram matrix of an $(\alpha a, \alpha^2 b, \alpha c)$-ETF in an orthogonal geometry of the same dimension.
Notice that the discriminant changes if the dimension is odd and $\alpha \notin \mathbb{F}_q^{\times 2}$.
If $n > d$ then $b \neq 0$, since the vectors of a $(0,0)$-equiangular system span a totally isotropic subspace (and therefore cannot form a frame), while those of an $(a,0)$-equiangular system with $a\neq 0$ are linearly independent (by the usual argument).
As $b$ is a quadratic residue, we can normalize to ensure $b = 1$ whenever $n > d$.
In that case the ambiguity associated with $\alpha = -1$ remains.
(We usually do not normalize $a$ or $c$ since each is often equal to $0$.)
After such normalization, the Gram matrix takes the form $G = S + aI$, where $S$ has zeros on the diagonal and $\pm 1$ off the diagonal.
Here, $S$ is known as the \textbf{signature matrix} of the rescaled ETF.
Up to a change in sign, it is uniquely determined by the unscaled ETF.
\end{remark}

We end our preliminary discussion with some trivial ETF constructions.

\begin{example}
For any $n\geq 2$ and any odd prime power $q = p^l$, the matrices $I,J \in \mathbb{F}_q^{n\times n}$ satisfy $I^2 = I$ and $J^2 = nJ$.
The former produces a $(1,0,1)$-ETF in an orthogonal geometry of dimension $d = n$ over $\mathbb{F}_q$.
The latter produces a $(1,1,n)$-ETF of size $1 \times n$.
When $p$ fails to divide $n$, we can take a Naimark complement to obtain an $(n-1,1,n)$-ETF of size $(n-1) \times n$.
\end{example}

\section{Field conditions}
\label{sec: fields}

In this section, we provide conditions under which the signature matrix of an ETF in one field produces another ETF in a different field.
Using this technique, we demonstrate a close correspondence between real ETFs and ETFs in finite orthogonal geometries.
In the finite field setting, we show that in most cases the parameters of an $(a,b,c)$-ETF belong to the base field $\mathbb{F}_p$, and so its Gram matrix factors to produce another ETF of the same size in an orthogonal geometry over the base field.
Considering relations between the parameters, we deduce necessary conditions on $d,n,p$ such that a $d\times n$ ETF exists in a finite orthogonal geometry of characteristic~$p$.

We use the following notation when moving between fields.
Let $R \subset \mathbb{C}$ be a unital ring of algebraic integers (that is, zeroes of monic polynomials with integer coefficients), and let $\pi \colon R \to F$ be a unital ring homomorphism into a finite field $F$ of characteristic~$p$.
(For instance, $R = \mathbb{Z}$ and $F = \mathbb{F}_p$.)
Given a matrix $M$ with entries in $R$, we denote $\overline{M} = \pi(M)$ for its entrywise image over $F$ and $\operatorname{rank}_F \overline{M}$ for its rank over $F$.
In case $M$ has integer entries, we also write $\operatorname{rank}_p \overline{M} := \operatorname{rank}_{\mathbb{F}_p} \overline{M} = \operatorname{rank}_F \overline{M}$.
We routinely abuse notation by identifying integers $k \in \mathbb{Z}$ with their images $\overline{k} = \pi(k) \in \mathbb{F}_p \leq F$.
The symbols $\equiv$ and $\equiv_p$ indicate equality in $F$ and $\mathbb{F}_p$, respectively, after appropriate applications of $\pi \colon R \to F$.
These abuses should not cause any confusion in practice, and the exact meaning should be clear from context.

The following proposition will prove indispensible.
It summarizes aspects of Propositions~13.3.1, 13.3.2, and 13.3.4 of~\cite{BH12}.
In~\cite{BH12} it is assumed that $R$ is the entire ring of algebraic integers in~$\mathbb{C}$.
However, a careful reading shows that the same proofs apply in our more general setting.

\begin{proposition}[\cite{BH12}]
\label{prop: p-ranks}
Let $M$ be an $n\times n$ matrix all of whose entries and eigenvalues lie in $R$.
\begin{itemize}
	\item[(a)]
	We have $\operatorname{rank}_F \overline{M} \leq \operatorname{rank} M$, with equality if $M$ is diagonalizable and every nonzero eigenvalue $\theta$ satisfies $\overline{\theta} \neq 0$.
	\item[(b)]
	Suppose $M$ is nonzero, singular, and symmetric with integral entries and constant row sum $\theta_0$, and that the $\theta_0$-eigenspace is one-dimensional.
	Denote the minimal polynomial of $M$ by $(x-\theta_0)h(x)$, and define $\epsilon = 1$ if $\overline{h(M)} = 0$ and $\epsilon = 0$ otherwise.
	If every other nonzero eigenvalue $\theta \neq \theta_0$ of $M$ satisfies $\overline{\theta} \neq 0$, then $\operatorname{rank}_F \overline{M} = \operatorname{rank} M - \epsilon$.
\end{itemize}
\end{proposition}

\subsection{Interactions between real and finite fields}

\begin{proposition}
\label{prop: project real ETF}
Suppose $S \in \mathbb{Z}^{n\times n}$ is the signature matrix of a real $d \times n$ ETF with $n > d+1$.
Let $q = p^l$ be any odd prime power, but if $n=2d$ assume there exists $\delta \in \mathbb{F}_q$ with $\delta^2 \equiv n-1$.
Then there exist $a,c \in \mathbb{F}_q$ such that $\overline{S} + aI \in \mathbb{F}_q^{n\times n}$ is the Gram matrix of an $(a,1,c)$-ETF of $n$ vectors in an orthogonal geometry on $\mathbb{F}_q^{d'}$, where $d' \leq d$ and
\[
a \equiv \begin{cases}
\sqrt{ \frac{ d(n-1) }{ n - d } }, & \text{if }n \neq 2d; \\
\delta, & \text{if }n = 2d
\end{cases}
\qquad
\text{and}
\qquad
c \equiv \begin{cases}
\sqrt{ \frac{ n^2 (n-1) }{d(n-d)} }, & \text{if }n \neq 2d; \\
2\delta, & \text{if }n = 2d.
\end{cases}
\]
(The square roots are integers in the indicated cases.)
Furthermore, $d' = d$ if $c \not\equiv 0$.
\end{proposition}

\begin{proof}
Let $G=S+\hat{a}I \in \mathbb{R}^{n\times n}$ be the Gram matrix of a $d \times n$ ETF with frame constant $\hat{c} \in \mathbb{R}$.
Since $G \neq 0$ satisfies $G^2 = \hat{c} G$, $S$ has eigenvalues $- \hat{a}$ and $\hat{c} - \hat{a}$, where
\[
\hat{a} = \sqrt{ \frac{ d(n-1) }{n-d} },
\qquad
\hat{c} = \tfrac{d}{n} \hat{a} = \sqrt{ \frac{ n^2(n-1)}{d(n-d)} }.
\]
(See \cite[Section~4]{STDH07}.)

First suppose $n \neq 2d$.
Corollary~13 of~\cite{STDH07} states that $S$ has integer eigenvalues, so $\hat{a},\hat{c} \in \mathbb{Z}$.
Therefore $G$ has integer entries and integer spectrum $\sigma(G) = \{ 0^{n-d}, \hat{c}^d \}$, where superscripts denote geometric multiplicities.
Its image $\overline{G} \in \mathbb{F}_q^{n\times n}$ satisfies $\overline{G}^2 = \hat{c} \overline{G}$, and, by Proposition~\ref{prop: orthogonal ETF from Gram matrix}, it is the Gram matrix of an $(\hat{a},1,\hat{c})$-ETF in an orthogonal geometry on $\mathbb{F}_q^{d'}$, where $d' = \operatorname{rank}_p \overline{G}$.
We have $d' \leq d$ by Proposition~\ref{prop: p-ranks}(a), and equality holds if $\hat{c} \not\equiv_p 0$.

Now suppose $n = 2d$.
We have $\hat{a} = \sqrt{n-1}$ and $\hat{c} = 2\sqrt{n-1}$, so the entries and eigenvalues of $G$ lie in the ring $R = \mathbb{Z}[\sqrt{n-1}]$.
Let $\pi \colon R \to \mathbb{F}_q$ be the unique unital ring homomorphism for which $\pi(\sqrt{n-1}) = \delta$.
In this case we consider $\overline{G} = \pi(G) \in \mathbb{F}_q^{n\times n}$.
As above, $\overline{G}$ is the Gram matrix of a $(\delta,1,2\delta)$-ETF in an orthogonal geometry on $\mathbb{F}_q^{d'}$, and $d' \leq d$ with equality if $\delta \neq 0$.
\end{proof}

We have a partial converse to Proposition~\ref{prop: project real ETF}.

\begin{proposition}
\label{prop: no wrap implies real}
Suppose there is an ETF of $n$ vectors in an orthogonal geometry on $\mathbb{F}_q^d$.
If $q = p^l$ is odd and $p > 2n-5$, then there is an ETF of $n$ vectors in $\mathbb{R}^d$.
\end{proposition}

We do not know if $p > 2n-5$ is a sharp bound.

\begin{proof}
By applying Proposition~\ref{prop: Naimark} if necessary, we may assume that $n \geq 2d$.
(Note that if the frame constant is zero then we already have $n \geq 2d$ by Proposition~\ref{prop: iso frames}.)
We may also assume that $n \geq d+2 \geq 4$, since otherwise a real ETF of the desired size is already known to exist.
Then it is easy to show 
\[ \max \{ d, n-d, n-1 \} \leq 2n - 5 < p, \]
so that $d \not\equiv_p 0$ and $n -d \not\equiv_p 0$.

Let $G = S + aI \in \mathbb{F}_q^{n\times n}$ be the Gram matrix of an $(a,1,c)$-ETF, where the signature matrix $S \in \mathbb{F}_q^{n\times n}$ has zeros on the diagonal and $\pm1$ off the diagonal.
Here $G^2 = c G = \tfrac{n}{d} a G$ by \eqref{eq: nadc}, or equivalently,
\[
S^2 = \tfrac{n-2d}{d} a S + \tfrac{n-d}{d} a^2 I.
\]
On the other hand, $S^2$ has entries in $\mathbb{F}_p$, and its diagonal is constantly $n-1$.
Therefore $\beta:= \tfrac{n-2d}{d}a$ lies in $\mathbb{F}_p$, and $\tfrac{n-d}{d} a^2 = n - 1$.
Consequently,
\begin{equation}
\label{eq: invert real 1}
\beta^2 = \frac{(n-2d)^2 (n-1)}{d(n-d)}. 
\end{equation}

We are going to lift $S$ to a matrix over $\mathbb{Z}$ and recover a real ETF.
Let $\pi \colon \mathbb{Z} \to \mathbb{F}_q$ be the unique unital ring homomorphism.
Define $\hat{\beta}$ to be the unique integer with $-\tfrac{1}{2}(p-1) \leq \hat{\beta} \leq \tfrac{1}{2}(p-1)$ and $\pi(\hat{\beta}) = \beta$, and define $\hat{S} \in \mathbb{Z}^{n\times n}$ to be the unique matrix with $\hat{S}_{ij} \in \{0, \pm1\}$ and $\pi(\hat{S}_{ij}) = S_{ij}$ for every $i,j$.
Since $\pi$ extends to a homomorphism of matrix rings, we have $\pi\left[ ( \hat{S}^2 )_{ij} - \hat{\beta} \hat{S}_{ij} \right] = 0$ whenever $i \neq j$.
On the other hand, when $i \neq j$,
\[
0
\leq
\left| ( \hat{S}^2 )_{ij} - \hat{\beta} \hat{S}_{ij} \right|
\leq
\left| ( \hat{S}^2 )_{ij} \right| + \left| \hat{\beta} \hat{S}_{ij} \right|
\leq (n-2) + |\hat{\beta}|
< p
\]
since $n < \tfrac{1}{2}(p+5)$ and $|\hat{\beta}| \leq \tfrac{1}{2}(p-1)$.
Therefore $( \hat{S}^2 )_{ij} - \hat{\beta} \hat{S}_{ij} = 0$ whenever $i \neq j$, and
\[
\hat{S}^2 = \hat{\beta} \hat{S} + (n-1) I.
\]
Hence $\hat{S}$ has just two eigenvalues, and one is negative since the trace is zero.
Writing $-\hat{a}$ for the negative eigenvalue, we see that $\hat{G} := \hat{S} + \hat{a} I$ is the Gram matrix of an ETF in $\mathbb{R}^{d'}$ for some $d'$.
Furthermore, proceeding as above we find
\begin{equation}
\label{eq: invert real 2}
\hat{\beta}^2 = \frac{(n-2d')^2 (n-1)}{d'(n-d')}.
\end{equation}

Arguing by cases, we now show $d' \in \{ d, n - d \}$.
First suppose $n \equiv_p 2d$.
Then~\eqref{eq: invert real 1} gives $\beta = 0$, so $\hat{\beta} = 0$ and $n = 2d'$ by~\eqref{eq: invert real 2}.
Since $n \geq 2d$ and $d \geq 2$ we have $0 \leq n - 2d \leq n-4 < p$.
It follows that $n = 2d$, hence $d' = d$.

Now suppose $n \not\equiv_p 2d$.
Rewriting~\eqref{eq: invert real 2}, we find that $d'$ is a root of
\[
h(x) = (4n-4 + \hat{\beta}^2)x^2 - n( \hat{\beta}^2 + 4n - 4)x + n^2(n-1) \in \mathbb{Z}[x].
\]
Likewise,~\eqref{eq: invert real 1} says that $d$ is a root of
\[
\overline{h}(x) = (4n-4 + \beta^2)x^2 - n( \beta^2 + 4n - 4)x + n^2(n-1) \in \mathbb{F}_p[x].
\]
By inspection of~\eqref{eq: invert real 1}, we see that the other root of $\overline{h}$ is $n-d$.
We have $d \not \equiv_p n-d$ by case assumption, so these are the only roots of $\overline{h}$ in $\mathbb{F}_p$.
Applying $\pi$ to $d'$ and the coefficients of $h$, we deduce that either $d' \equiv_p d$ or $d' \equiv_p n-d$.
In the first case, observe
\[ 0 \leq | d - d' | = \max \{ d - d', d' - d \} \leq n - 2 < p, \]
so that $d' = d$.
In the second case, we have $n - ( d + d') \leq n -2 < p$ and $(d + d') - n \leq d < p$, so $0 \leq | n - d - d' | < p$ and $d' = n - d$.
Now taking a Naimark complement produces a real $d \times n$ ETF.
\end{proof}

\subsection{Additional constraints}

The following theorem and its corollary may be seen as finite field versions of~\cite[Theorem 15]{STDH07}.
Our proof uses a similar strategy.

\begin{theorem}
\label{thm: orth ETF in Fp}
Suppose there is an $(a,1,c)$-ETF of $n$ vectors in an orthogonal geometry on $\mathbb{F}_q^d$, where $q = p^l$ is an odd prime power.
Then $a,c \in \mathbb{F}_{p^j}$ for some $j \in \{1,2\}$, with $j =1$ if $n \neq 2d$ or $c = 0$.
Consequently, there is an ETF of $n$ vectors in an orthogonal geometry on $\mathbb{F}_{p^j}^d$.
\end{theorem}

We emphasize that the hypothesis $n \neq 2d$ is not to be confused with the stronger condition $n \not\equiv_p 2d$.

\begin{proof}
Let $G \in \mathbb{F}_q^{n\times n}$ be the Gram matrix of the given ETF, and define $S = G - aI$.
Notice that $S$ has entries in $\mathbb{F}_p$, and so
\begin{equation}
\label{eq: chiS1}
\chi_S(x) := \det(xI - S) \in \mathbb{F}_p[x].
\end{equation}
On the other hand, $G$ has minimal polynomial $m_G(x) = x(x-c)$ with only $0$ and $c$ as its roots.
If $c \neq 0$, then the Jordan normal form $D$ of $G$ is diagonal with $d = \operatorname{rank} G$ copies of $c$ and $n-d$ copies of $0$ on the diagonal, and if $c = 0$ then $D$ is upper triangular with all zeros on the diagonal.
In either case, we find that $\chi_G(x) := \det(xI - G) = (x-c)^d x^{n-d}$, and so
\begin{equation}
\label{eq: chiS2}
\chi_S(x) = \chi_G(x+a) = (x+a-c)^d(x+a)^{n-d}.
\end{equation}
By considering the Jordan normal form of $G$ again, we see that $d \neq n$.
(If $c=0$ then $d \leq \tfrac{1}{2}n < n$.
If $c \neq 0$ and $d = n$, then $D = cI = G$, despite being the Gram matrix of an $(a,1,c)$-ETF.)
Therefore $\chi_S(-a) = 0$.

Let $h \in \mathbb{F}_p[x]$ be the minimal polynomial of $-a$ over $\mathbb{F}_p$. 
By definition, this is the unique monic polynomial in $\mathbb{F}_p[x]$ that satisfies $h(-a) = 0$ and divides every other polynomial $k \in \mathbb{F}_p[x]$ with $k(-a) = 0$. 
Equivalently, it is the unique irreducible monic polynomial $h(x) \in \mathbb{F}_p[x]$ satisfying $h(-a) = 0 $. 
Using~\eqref{eq: chiS1} and~\eqref{eq: chiS2}, we see that $h(x)$ divides $\chi_S(x)$ in $\mathbb{F}_p[x]$, and that it has no roots in an algebraic closure of $\mathbb{F}_p$ other than $-a$ and, perhaps, $c-a$.
Since finite fields are perfect, $h(x)$ has no repeated roots, and it must be one of two possibilities:
\begin{itemize}
\item[(I)]
$h(x) = x+ a$, or
\item[(II)]
$h(x) = (x+ a)(x + a - c)$.
\end{itemize}
In both cases, $-a$ belongs to a quadratic extension of $\mathbb{F}_p$, and so $a \in \mathbb{F}_{p^2}$.
Considering that $G \neq 0$ has entries in $\mathbb{F}_{p^2}$ and satisfies $G^2 = cG$, we deduce that $c \in \mathbb{F}_{p^2}$ as well.
Moreover, if $c = 0$ then option (II) cannot occur since $h(x)$ has no repeated roots, and in this case we conclude that $a \in \mathbb{F}_p$.

Now suppose $n \neq 2d$ and $c \neq 0$.
Assume for the sake of contradiction that option (II) occurs.
Then $h(x)$ is also the minimal polynomial of $c-a$ over $\mathbb{F}_p$. 
In the factorization
\[ \chi_S(x) = \prod_{i=1}^k f_i(x) \]
of $\chi_S(x)$ as a product of monic irreducible polynomials $f_i(x) \in \mathbb{F}_p[x]$, every factor $f_i(x)$ has either $-a$ or $c-a$ as a root by~\eqref{eq: chiS2}. 
By uniqueness of the minimal polynomial, every factor $f_i(x)$ equals $h(x)$, so that
\[ \chi_S(x) = h(x)^k = (x + a)^k (x + a - c)^k. \]
This contradicts~\eqref{eq: chiS2} since $a \neq a - c$ and $d \neq n-d$.
Therefore option (I) holds, and $a,c \in \mathbb{F}_p$.

In every case, the ``consequentally'' statement follows from Proposition~\ref{prop: orthogonal ETF from Gram matrix}.
\end{proof}

\begin{corollary}
\label{cor: integrality}
Suppose there is an ETF of $n \notin \{d, 2d\}$ vectors in an orthogonal geometry on $\mathbb{F}_q^d$, where $q = p^l$ is an odd prime power.
If $d \not \equiv_p 0$ and $n \not\equiv_p 1$, then $n \not\equiv_p d$ and $\overline{d(n-1)}(\overline{n-d})^{-1} \in \mathbb{F}_p^{\times 2}$.
\end{corollary}

\begin{proof}
As explained in Remark~\ref{rem: rescale Gram}, there is an $(a,1,c)$-ETF of $n$ vectors in an orthogonal geometry on $\mathbb{F}_q^d$, where $a \in \mathbb{F}_p$ by Theorem~\ref{thm: orth ETF in Fp}.
Proposition~\ref{prop: Welch} states that $1 = (\overline{n-d}) [ \overline{d(n-1)} ]^{-1} a^2$, so $\overline{n-d} \neq 0$ and 
\[ \overline{d(n-1)}(\overline{n-d})^{-1} = a^2 \in \mathbb{F}_p^{\times 2}. \qedhere \]
\end{proof}

\section{Relationships with modular SRGs}
\label{sec: SRG}

As in the real setting, ETFs in finite orthogonal geometries are closely associated with graphs given by the sign patterns in their signature matrices.
In this section we study this correspondence in detail.
Specifically, let $G$ be the Gram matrix of an $(a,b,c)$-ETF $\Phi$ of $n>d$ vectors in a $d$-dimensional orthogonal geometry over $\mathbb{F}_q$, where $q$ is odd.
After rescaling as in Remark~\ref{rem: rescale Gram}, we may assume $b = 1$.
Then the signature matrix $S := G - a I$ of $\Phi$ is symmetric with zeros on the diagonal and $\pm 1$ off the diagonal, and so it is the Seidel adjacency matrix of a graph $\Gamma$ on $n$ vertices.
A second graph related to $\Phi$ is obtained by multiplying each of the frame vectors by $\pm 1$ to produce another $(a,1,c)$-ETF $\Phi'$ with Gram matrix
\[
G' =
\left[ \begin{array}{cc} a & \mathbf{1}_v^\top \\ \mathbf{1}_v & \Sigma + a I_v \end{array} \right].
\]
Here $\Phi'$ is said to have \textbf{normalized signature matrix}, and $\Sigma$ is the Seidel adjacency matrix of a graph $\Gamma'$ on $n-1$ vertices.

In this section we characterize graphs $\Gamma'$ for which there exists $a$ such that $G'$ above is the Gram matrix of an ETF (Theorem~\ref{thm: Seidel Waldron}).
We also characterize graphs $\Gamma$ for which there exists $a$ such that $G$ above is the Gram matrix of an ETF and the all-ones vector is an eigenvector of $G$ (Theorem~\ref{thm: centered or axial}).
In both cases, we are led to a modular generalization of strongly regular graphs.
SRGs form a special subclass, and for this subclass we identify the ETF size and parameters directly from the SRG parameters (Theorem~\ref{thm: 2-graph*} and Theorem~\ref{thm: 2-graph}).

We frequently encounter matrices $S \in \{ 0, 1, -1 \}^{n\times n}$, and we sometimes abuse notation by using the same symbol for $S \in \mathbb{Z}^{n\times n}$ as well as its image $\overline{S} \in \mathbb{F}_q^{n\times n}$.
This should not cause any confusion in practice.

\subsection{ETFs with normalized signature matrix}
\begin{definition}
\label{def: mod SRG}
Given an odd prime $p$, a graph $\Gamma$ on vertex set $V$ of size $v$ is called a \textbf{$p$-modular strongly regular graph} with parameters $v,k,\lambda,\mu$ (briefly, a $(v,k,\lambda,\mu)\text{-SRG}_p$) if the following hold:
 	\begin{itemize}
	\item[(i)]
	every vertex $i \in V$ has valency $k_i \equiv_p k$,
	\item[(ii)]
	whenever $i,j\in V$ are adjacent, they have exactly $\lambda_{ij} \equiv_p \lambda$ neighbors in common,
	\item[(iii)]
	whenever $i,j \in V$ are distinct and non-adjacent, they have exactly ${\mu_{ij} \equiv_p \mu}$ neighbors in common.
	\end{itemize}

\end{definition}

Equivalently, $\Gamma$ is a $(v,k,\lambda,\mu)\text{-SRG}_p$ if and only if its $\{0,1\}$-adjacency matrix $A$ satisfies $A^2 \equiv_p \mu J + (\lambda - \mu)A + (k-\mu)I$.
Applying the operators on both sides of the equivalence to the vector $\mathbf{1}$, we see that the parameters of a $(v,k,\lambda,\mu)\text{-SRG}_p$ necessarily satisfy
\begin{equation}
\label{eq:SRGp params}
k(k-\lambda-1) \equiv_p \mu(v-k-1).
\end{equation}

The following characterization of modular SRGs can be proved easily using~\eqref{eq: Seidel from 01} and the inversion of a $4\times 4$ matrix.
We omit details.

\begin{lemma}
\label{lem: SRGp sig}
Let $\Sigma$ be the Seidel adjacency matrix of a graph $\Gamma$ on $v$ vertices.
The following are true for any choice of odd prime $p$:
	\begin{itemize}
	\item[(a)]
	If $\Gamma$ is a $(v,k,\lambda,\mu)\text{-SRG}_p$, then $\Sigma \mathbf{1} \equiv_p (v-2k-1)\mathbf{1}$ and
	\[
	\Sigma^2 \equiv_p (v-4k+2\lambda+2\mu)J + 2(\mu - \lambda -1)\Sigma + (4k - 2\lambda - 2\mu - 1)I.
	\]
	\item[(b)]
	Suppose there exist $\alpha,\beta,\gamma,\theta$ in an extension field $\mathbb{F}_q \geq \mathbb{F}_p$ such that
	$\Sigma \mathbf{1} \equiv \theta \mathbf{1}$ and $\Sigma^2 \equiv \alpha J + \beta \Sigma + \gamma I$ in $\mathbb{F}_q$.
	Then $\alpha,\beta,\gamma,\theta \in \mathbb{F}_p$, and $\Gamma$ is a $(v,k,\lambda,\mu)\text{-SRG}_p$ for
	\begin{equation}
	\label{eq:vklmu}
	\left\{
	\begin{array}{rcl}
	v & \equiv & \alpha + \gamma + 1 \\[4 pt]
	k & \equiv &  \tfrac{1}{2}( \alpha + \gamma - \theta) \\[4 pt]
	\lambda & \equiv & \tfrac{1}{4}(2 \alpha - \beta + \gamma - 2 \theta - 3) \\[4 pt]
	\mu & \equiv & \tfrac{1}{4}( 2 \alpha + \beta + \gamma - 2 \theta + 1),
	\end{array}
	\right.
	\end{equation}
	where the inverses are taken in $\mathbb{F}_p$.
	\end{itemize}
\end{lemma}

The theorem below may be considered as a $p$-modular version of Seidel's classical theorem relating strongly regular graphs with real ETFs~\cite{S91,T77,W09}.

\begin{theorem}
\label{thm: Seidel Waldron}
Let $\Sigma$ be the Seidel adjacency matrix of a graph $\Gamma$ on $v = n-1$ vertices, and define
\[
S = 
\left[
\begin{array}{cc}
0 & \mathbf{1}_v^\top \\
\mathbf{1}_v & \Sigma
\end{array}
\right].
\]
Given an odd prime power $q = p^l$, there exist $a,c \in \mathbb{F}_q$ such that ${S+aI} \in \mathbb{F}_q^{n \times n}$ is the Gram matrix of an $(a,1,c)$-ETF if and only if $\Gamma$ is a $(v,k,\lambda,\mu)\text{-SRG}_p$ and all of the following hold in $\mathbb{F}_q$:
	\begin{itemize}
	\item[(i)]
	$k \equiv_p 2 \mu$,
	\item[(ii)]
	$v \equiv_p 3k - 2\lambda - 1$,
	\item[(iii)]
	there exists $\delta \in \mathbb{F}_q$ such that $\delta^2 \equiv (\lambda - \mu)^2 + 4(k-\lambda)$.
	\end{itemize}
In that case, $S + aI$ is the Gram matrix of an $(a,1,c)$-ETF if and only if $a = \lambda - \mu + \epsilon \delta + 1$ and $c = 2 \epsilon \delta$ for some $\epsilon \in \{ \pm 1 \}$.
\end{theorem}

If $\mu \not\equiv_p 0$, then condition (i) above implies (ii) for any $(v,k,\lambda,\mu)\text{-SRG}_p$, as seen from~\eqref{eq:SRGp params}.
Similarly, (ii) implies (i) if $k - \lambda - 1 \not\equiv_p 0$.
However, (i) and (ii) are independent in general.

\begin{proof}
First we make a reduction.
For any choice of $a,c\in \mathbb{F}_q$, the matrix
\[
G 
:= S + a I
= \left[ \begin{array}{cc}
a & \mathbf{1}_v^\top \\
\mathbf{1}_v & \Sigma + aI_v
\end{array} \right]
\in \mathbb{F}_q^{n\times n}
\]
satisfies $G^2 = cG$ if and only if
\[
\left[ \begin{array}{cc}
a^2 + v & \mathbf{1}_v^\top(\Sigma + 2aI_v) \\
(\Sigma + 2aI_v) \mathbf{1}_v & \Sigma^2 + J_v + 2a\Sigma + a^2 I_v
\end{array} \right]
=
\left[ \begin{array}{cc}
ac & c\mathbf{1}_v^\top \\
c\mathbf{1}_v & c(\Sigma + aI_v)
\end{array} \right].
\]
This occurs if and only if:
	\begin{itemize}
	\item[(I)]
	$\Sigma \mathbf{1}_v \equiv (c-2a)\mathbf{1}_v$, 
	\item[(II)]
	$\Sigma^2 \equiv -J_v + (c-2a)\Sigma + (ac-a^2)I_v$,
	\item[(III)]
	$v \equiv ac-a^2$.
	\end{itemize}
In any case $\Sigma^2$ has constant diagonal $v-1$, so that (II) is equivalent to the combination of (II) and (III).
By Proposition~\ref{prop: orthogonal ETF from Gram matrix}, $G$ is the Gram matrix of an $(a,1,c)$-ETF if and only if (I) and (II) hold.

In the forward direction, suppose there exist $a,c\in \mathbb{F}_q$ such that $G$ is the Gram matrix of an $(a,1,c)$-ETF.
We prove (i)--(iii) and deduce the values of $a$ and $c$.
Since (I) and (II) hold, $\Sigma \mathbf{1}_v \equiv \theta \mathbf{1}_v$ and $\Sigma^2 \equiv \alpha J_v + \beta \Sigma + \gamma I_v$ for
\begin{equation}
\label{eq: alpha beta theta gamma}
\alpha = -1, \quad \beta = \theta = c - 2a, \quad \gamma = ac-a^2.
\end{equation}
By Lemma~\ref{lem: SRGp sig}(b), $\Gamma$ is a $(v,k,\lambda,\mu)\text{-SRG}_p$ with parameters as in~\eqref{eq:vklmu}.
In particular,
\[
2(k - 2\mu) \equiv (\theta - \beta) - (\alpha + 1) \qquad \text{and} \qquad 2(v-3k+2\lambda+1) \equiv (\theta-\beta) + (\alpha +1).
\]
Since $\theta = \beta$ and $\alpha = -1$, we have 
\begin{equation}
\label{eq: i and ii}
k \equiv_p 2 \mu \qquad \text{and}  \qquad v \equiv_p 3k - 2\lambda - 1 \equiv_p 6 \mu - 2 \lambda - 1.
\end{equation}
This gives (i) and (ii).
For (iii), observe from~\eqref{eq: alpha beta theta gamma} and~\eqref{eq:vklmu} that $c = \beta + 2a$ and
\[
v \equiv \gamma = ac - a^2 = a^2 + \beta a.
\]
Completing the square to solve for $a$, we find that
\begin{equation}
\label{eq: Delta square}
4(a + \tfrac{\beta}{2})^2 \equiv 4v + \beta^2 =: \Delta.
\end{equation}
Another application of~\eqref{eq:vklmu} gives $\beta \equiv 2(\mu - \lambda - 1)$, so by~\eqref{eq: i and ii}, the right-hand side of~\eqref{eq: Delta square} is
\[
\Delta \equiv 4(6 \mu - 2\lambda - 1) + 4(\mu - \lambda - 1)^2 \equiv 4[ (\lambda - \mu)^2 + 4(k-\mu) ].
\]
Therefore $a + \tfrac{\beta}{2} \in \mathbb{F}_q$ satisfies $(\alpha + \tfrac{\beta}{2})^2 \equiv \tfrac{1}{4}\Delta \equiv (\lambda - \mu)^2 + 4(k - \mu)$.
This gives~(iii).
Additionally, for either choice of $\delta \in \mathbb{F}_q$ such that $\delta^2  \equiv (\lambda - \mu)^2 + 4(k - \mu)$ we have $a = - \tfrac{\beta}{2} + \epsilon \delta \equiv \lambda - \mu + \epsilon \delta + 1$ for some $\epsilon \in \{ \pm 1 \}$, while $c = 2a + \beta = 2 \epsilon \delta$.

In the reverse direction, suppose $\Gamma$ is a $(v,k,\lambda,\mu)\text{-SRG}_p$ that satisfies (i)--(iii).
Choose $\delta \in \mathbb{F}_q$ such that $\delta^2 \equiv (\lambda - \mu)^2 + 4(k - \mu)$.
Fix $\epsilon \in \{ \pm 1 \}$, and define $a = \lambda -\mu + \epsilon \delta + 1$ and $c = 2 \epsilon \delta$.
In order to prove that $G$ is the Gram matrix of an $(a,1,c)$-ETF, we show~(I) and~(II).
By Lemma~\ref{lem: SRGp sig}(a), $\Sigma \mathbf{1}_v \equiv \theta \mathbf{1}_v$ and $\Sigma^2 \equiv \alpha J_v + \beta \Sigma + \gamma I_v$, where 
\begin{align*}
\alpha &\equiv v - 4k + 2\lambda + 2\mu \equiv -1, \\
\beta &\equiv 2(\mu - \lambda - 1), \\
\gamma &\equiv 4k - 2\lambda - 2\mu - 1 \equiv v, \\
\theta &\equiv v - 2k - 1 \equiv 2(\mu - \lambda - 1).
\end{align*}
We have $c - 2a \equiv 2(\mu - \lambda -1) \equiv \beta \equiv \theta$, which gives (I).
For (II), observe as above that $4(a + \tfrac{\beta}{2})^2 = 4 \delta^2 \equiv 4v + \beta^2$, hence $a^2 + \beta \equiv v$.
Consequently,
\[
ac - a^2 = a^2 + \beta a \equiv v \equiv \gamma.
\]
Comparing coefficients, we find that (II) holds.
\end{proof}

\begin{theorem}
\label{thm: 2-graph*}
Let $\Sigma$ be the Seidel adjacency matrix of a nontrivial SRG with $f \neq g$.
Define
\[
G 
= \left[ \begin{array}{cc} 2r+1 & \mathbf{1}_v^\top \\ \mathbf{1}_v & \Sigma + (2r+1)I \end{array} \right]
\in \mathbb{Z}^{n\times n},
\]
where $n = v+1$.
Suppose $p$ is an odd prime that divides both $k - 2\mu$ and $v - 3k + 2\lambda + 1$.
Then $\overline{G} \in \mathbb{F}_p^{n\times n}$ is the Gram matrix of a $(2r+1,1,2r-2s)$-ETF in an orthogonal geometry on $\mathbb{F}_p^d$, where 
\[ d = \operatorname{rank}_p \overline{\Sigma +(2r+1) I}. \]
Furthermore, $d = g+1$ if $r \not\equiv_p s$, whereas $d \leq \min\{ f+1, g+1 \}$ if $r \equiv_p s$.
\end{theorem}

We omit the case $f = g$ for simplicity, since in that case the matrix $G$ above is already the Gram matrix of a real ETF with $n = 2d$, and Proposition~\ref{prop: project real ETF} applies.

\begin{proof}
Recall our notation~\eqref{eq: SRG sig eigs} for the eigenvalues of $\Sigma$, where $2r+1 = -\theta_r$.
Theorem~\ref{thm: Seidel Waldron} implies that $\overline{G}$ is the Gram matrix of an $(a,b,c)$-ETF with the given parameters, and the dimension is $d=\operatorname{rank}_p \overline{G}$ by Proposition~\ref{prop: orthogonal ETF from Gram matrix}.
It remains to prove that $\operatorname{rank}_p \overline{G} = \operatorname{rank}_p \overline{\Sigma -\theta_r I}$, and to deduce the ``furthermore'' statements.

To begin, we argue in cases to show $\operatorname{rank}_p \overline{G} = \operatorname{rank}_p \overline{\Sigma -\theta_r I}$.
First suppose $\theta_k \not\equiv_p \theta_r$.
Then $x := \overline{(\theta_k - \theta_r)}^{-1} \overline{\mathbf{1}_v} \in \mathbb{F}_p^v$ satisfies $\overline{ (\Sigma - \theta_r I) }x = \overline{ \mathbf{1}_v }$.
Taking transposes, we see that $\overline{ \mathbf{1}_v^\top }$ lies in the row space of $\overline{ \Sigma - \theta_r I }$, and so
\[
\operatorname{rank}_p \overline{ \Sigma  - \theta_r I}
=
\operatorname{rank}_p \overline{ \left[ \begin{array}{c}  \mathbf{1}_v^\top \\ \Sigma - \theta_r I \end{array} \right] }.
\]
If we can show that $\overline{ \mathbf{1}_v^\top x } \equiv -\theta_r$, then we can continue to find
\[
\operatorname{rank}_p \overline{ \left[ \begin{array}{c}  \mathbf{1}_v^\top \\ \Sigma - \theta_r I \end{array} \right] }
=
\operatorname{rank}_p \overline{ \left[ \begin{array}{cc} - \theta_r & \mathbf{1}_v^\top \\ \mathbf{1}_v & \Sigma - \theta_r I \end{array} \right] },
\]
where the additional column on the right already lies in the column space on the left, as the image of $x$.
It will follow that $\operatorname{rank}_p \overline{ \Sigma - \theta_r I} = \operatorname{rank}_p \overline{G}$ as desired.
In fact, the definition of $x$ gives $\overline{ \mathbf{1}_v^\top x } = \overline{(\theta_k - \theta_r)}^{-1} \overline{v}$, so we need only show $v \equiv_p -\theta_r(\theta_k - \theta_r)$.
In other words, to prove $\operatorname{rank}_p \overline{G} = \operatorname{rank}_p \overline{\Sigma -\theta_r I}$ it suffices in this case to show
\begin{equation}
\label{eq: 2-graph* rel 1}
v \equiv_p (2r+1)(v-2k+2r).
\end{equation}
By applying the relation $r^2 = (\lambda-\mu)r + (k-\mu)$ from~\eqref{eq: eig quad}, we find that
\begin{align*}
(2r+1)(v-2k+2r) 
&= 4r^2 + (2v-4k+2)r + v-2k \\
&= (2v-4k+4\lambda-4\mu+2)r + v + 2k - 4\mu.
\end{align*}
Substituting $v \equiv_p 3k - 2\lambda - 1$, we see that the coefficient on $r$ above is congruent to $2k - 4\mu$.
Since $k \equiv_p 2 \mu$ we deduce~\eqref{eq: 2-graph* rel 1}.
Therefore, $\operatorname{rank}_p \overline{G} = \operatorname{rank}_p \overline{ \Sigma -\theta_r I}$ when $\theta_k \not\equiv_p \theta_r$.

Next we prove that $\operatorname{rank}_p \overline{G} = \operatorname{rank}_p \overline{\Sigma -\theta_r I}$ in the case $\theta_k \equiv_p \theta_r$.
Eventually we will find a suitable choice for $x$ to use the argument above.
In the meantime, observe that, by~\eqref{eq: SRG eig rels 2} and our hypotheses,
\[
\tfrac{1}{v}(\theta_k - \theta_r)(v - 2k + 2s) = v -4k + 2\lambda + 2\mu \equiv_p -1.
\]
Rewriting the left-hand side, we find
\[
-1 \equiv_p \theta_k - \theta_r - \frac{ 2(\theta_k - \theta_r)(k-s) }{v},
\]
the fraction being an integer.
By assumption $\theta_k - \theta_r \equiv_p 0$, thus
\begin{equation}
\label{eq: 2-graph* rel 2}
\frac{ 2(\theta_k - \theta_r)(k-s) }{v} \in \mathbb{Z} \setminus p \mathbb{Z},
\end{equation}
and in particular $\theta_k - \theta_r \neq 0$.

Now we can produce $x$.
Working over $\mathbb{R}$ for the moment, recall that $\operatorname{ker} (\Sigma - \theta_rI)$ contains the column space of $(r-s)E_r = A-sI-(k-s)v^{-1} J$.
For an integer $m$ yet to be determined, define vectors $y = \left[ \begin{array}{cccc} m & 0 & \dotso & 0 \end{array} \right]^\top$ and $x = \tfrac{1}{\theta_k - \theta_r} \mathbf{1}_v + (r-s)E_r y$ in $\mathbb{R}^v$.
Simplifying, we find
\begin{equation}
\label{eq: vector x}
x =  (A-sI) y + \frac{ v - (\theta_k - \theta_r)(k-s)m}{v(\theta_k - \theta_r)} \mathbf{1}_v.
\end{equation}
Our goal is to choose $m$ in such a way that all factors of $p$ are cleared from the denominator of~\eqref{eq: vector x}.
In that case, we can interpret $x$ as a vector over $\mathbb{F}_p$, and the same relations that helped above will again hold.
Let $p^{\ell_1}$, $p^{\ell_2}$, and $p^{\ell_3}$ be the largest powers of $p$ that divide $v$, $\theta_k - \theta_r$, and $k-s$, respectively.
Then $p^{\ell_1 + \ell_2}$ is the largest power of $p$ that occurs in the denominator of~\eqref{eq: vector x}, and \eqref{eq: 2-graph* rel 2}~implies $\ell_1 = {\ell_2 + \ell_3}$.
Hence,
\[ 
\operatorname{gcd}\Bigl(p^{\ell_1 + \ell_2}, (\theta_k - \theta_r)(k-s)\Bigr) 
= \operatorname{gcd}\bigl(p^{\ell_1 + \ell_2}, p^{\ell_1} \bigr) 
= p^{\ell_1}
\]
divides $v$.
By Bezout's identity, there is an integer~$m$ for which $v - (\theta_k - \theta_r)(k-s)m$ is a multiple of $p^{\ell_1 + \ell_2}$.
With $m$ so defined, the denominator of~\eqref{eq: vector x} clears as desired.
We have $(\Sigma - \theta_r I) x = \mathbf{1}_v$ and $\mathbf{1}_v^\top x = \tfrac{v}{\theta_k - \theta_r}$, since $(r-s)E_r y$ lies in the kernel of $\Sigma - \theta_r I$ and the orthogonal complement of $\mathbf{1}_v$.
Additionally, all factors of $p$ clear from the denominator of $\tfrac{v}{\theta_k - \theta_r}$ by~\eqref{eq: 2-graph* rel 2}.
Overall, we may interpret $x$ as a vector in $\mathbb{F}_p^v$ for which the desired relations hold, and the same argument as above shows that $\operatorname{rank}_p \overline{G} = \operatorname{rank}_p \overline{ \Sigma - \theta_r I}$.

It remains to prove our ``furthermore'' statements about $d = \operatorname{rank}_p \overline{ \Sigma - \theta_r I}$.
We are going to use Proposition~\ref{prop: p-ranks}.
Define $M = {\Sigma + (2r+1) I} \in \mathbb{R}^{v \times v}$, so that $d = \operatorname{rank}_p \overline{M}$.
Here $M$ has eigenvalues $\theta_k - \theta_r$, $0$, $\theta_s - \theta_r$.
By~\eqref{eq: SRG S1} it satisfies
\begin{equation}
\label{eq: MJ}
M\bigl( M + (\theta_r - \theta_s)I \bigr) = (v - 4k + 2\lambda + 2\mu)J.
\end{equation}
Applying our hypotheses $v \equiv_p 3k-2\lambda - 1$ and $2\mu \equiv_p k$, we find
\begin{equation}
\label{eq: 2-graph* rel 3}
\overline{ M\Bigl( M + (\theta_r - \theta_s)I \Bigr) } = - \overline{J} \neq 0.
\end{equation}
Since $MJ  = (\theta_k - \theta_r)J$, we deduce from~\eqref{eq: MJ} and~\eqref{eq: 2-graph* rel 3} that $M$ has minimal polynomial
\[
x(x+ \theta_r - \theta_s)(x + \theta_r - \theta_k) =: (x+\theta_r - \theta_k) h(x)
\]
and spectrum $\sigma(M) = \{ (\theta_k - \theta_r)^1, 0^f, (\theta_s - \theta_r)^g \}$, with multiplicities denoted by superscripts and all three eigenvalues distinct.

We argue in cases.
If $r \not\equiv_p s$ and $\theta_k - \theta_r \not\equiv_p 0$, then $\theta_s - \theta_r \not\equiv_p 0$ and Proposition~\ref{prop: p-ranks}(a) immediately gives $d = g+1$.
If $r \not\equiv_p s$ and $\theta_k - \theta_r \equiv_p 0$, then Proposition~\ref{prop: p-ranks}(b) applies to produce $d = g + 1 - \epsilon$, where $\epsilon = 0$ since $\overline{h(M)} \neq 0$ by~\eqref{eq: 2-graph* rel 3}.
Hence $d = g + 1$ in this case as well.
Finally, suppose $r \equiv_p s$.
Then $\theta_r \equiv_p \theta_s$, and $\overline{M} = \overline{ \Sigma - \theta_r I } = \overline{ \Sigma - \theta_s I}$.
By Proposition~\ref{prop: p-ranks}(a), $\operatorname{rank}_p \overline{M} \leq \operatorname{rank} (\Sigma - \theta_r I) = g +1$ and also $\operatorname{rank}_p \overline{M} \leq \operatorname{rank} (\Sigma - \theta_s I) = f +1$.
This completes the proof.
\end{proof}

\subsection{ETFs with centroidal symmetry}
So far we have related graphs on $v$ vertices with ETFs having $n = v+1$ vectors.
Now we consider ETFs with $n = v$ vectors.
Specifically, we characterize ETFs having \textbf{centroidal symmetry}, in the sense that the all-ones vector is an eigenvector of the Gram matrix~\cite{FJMPW18}.

\begin{theorem}
\label{thm: centered or axial}
Let $\Sigma$ be the Seidel adjacency matrix of a graph $\Gamma$ on $v$ vertices.
The following are equivalent for any choice of odd prime power $q = p^l$:
	\begin{itemize}
	\item[(i)]
	there exist $a,c,\theta \in \mathbb{F}_q$ such that $G = \Sigma + aI$ is the Gram matrix of an $(a,1,c)$-ETF, and $G\mathbf{1} \equiv \theta \mathbf{1}$,
	\smallskip
	\item[(ii)]
	$\Gamma$ is a $(v,k,\lambda,\mu)\text{-SRG}_p$ with $v \equiv_p 4k - 2\lambda - 2\mu$, and there exists $\delta \in \mathbb{F}_q$ such that $\delta^2 \equiv (\lambda - \mu)^2 + 4(k-\mu)$.
	\end{itemize}
In that case, $\Sigma + aI$ is the Gram matrix of an $(a,1,c)$-ETF if and only if $a \equiv \lambda - \mu + \epsilon \delta + 1$ and $c \equiv 2 \epsilon \delta$ for some $\epsilon \in \{ \pm 1 \}$.
\end{theorem}

\begin{proof}
Suppose (i) holds.
Then $\Sigma \mathbf{1} = (\theta-a) \mathbf{1}$, and the equation $G^2 = cG$ produces $\Sigma^2 = (c-2a) \Sigma + a(c-a)I$.
From Lemma~\ref{lem: SRGp sig} and the linear independence of $J,\Sigma,I$, we deduce that $\Gamma$ is a $(v,k,\lambda,\mu)\text{-SRG}_p$ with $v - 4k + 2\lambda + 2\mu \equiv 0$, $2(\mu - \lambda - 1) \equiv c-2a$, and $4k - 2\lambda - 2\mu - 1 \equiv a(c-a)$.
In particular, $c = 2( a + \mu - \lambda - 1)$ and
\begin{align*}
4k - 2\lambda - 2\mu - 1 &\equiv a(a+ 2\mu - 2\lambda - 2) \\
&\equiv [ (a + \mu - \lambda - 1) - (\mu - \lambda - 1)][ (a + \mu - \lambda - 1) + (\mu - \lambda - 1) ] \\
&\equiv (a+\mu - \lambda - 1)^2 - (\lambda - \mu + 1)^2.
\end{align*}
Rearranging, we find
\[
(a  +\mu - \lambda - 1)^2 \equiv (\lambda - \mu)^2 + 4(k - \mu).
\]
Hence there exists $\delta \in \mathbb{F}_q$ such that $\delta^2 = (\lambda - \mu)^2 + 4(k-\mu)$, and for either choice of such $\delta$ there exists $\epsilon \in \{ \pm  1\}$ such that $a + \mu - \lambda - 1 = \epsilon \delta$ and $c = 2 \epsilon \delta$.

Now assume~(ii).
Pick $\epsilon \in \{ \pm 1 \}$, and define $a = \lambda - \mu + \epsilon \delta + 1$ and $c = 2\epsilon \delta$.
Then $c -2a \equiv 2(\mu - \lambda - 1)$ and 
\[ a(c-a) \equiv [\epsilon \delta + \lambda - \mu - 1][ \epsilon \delta - (\lambda - \mu - 1)] \equiv \delta^2 - (\lambda - \mu + 1)^2 \equiv 4k - 2\lambda - 2\mu - 1. \]
From these identities, the assumption $v - 4k + 2\lambda + 2\mu \equiv_p 0$, and Lemma~\ref{lem: SRGp sig}(a), we conclude that $\Sigma^2 \equiv (c-2a) \Sigma + a(c-a) I$.
For $G = \Sigma + aI$, this says $G^2 = cG$.
Finally, Lemma~\ref{lem: SRGp sig}(a) supplies $\theta \in \mathbb{F}_q$ such that $\Sigma \mathbf{1} = (\theta-a) \mathbf{1}$, that is, $G \mathbf{1} = \theta \mathbf{1}$.
\end{proof}

\begin{theorem}
\label{thm: 2-graph}
Let $\Sigma$ be the Seidel adjacency matrix of a nontrivial SRG with $f \neq g$, and define $G = \Sigma + (2r+1) I$.
Suppose $p$ is an odd prime that divides $v - 4k + 2\lambda + 2\mu$.
Then $\overline{G}$ is the Gram matrix of a $(2r+1,1,2r-2s)$-ETF with $n=v$ vectors in an orthogonal geometry on $\mathbb{F}_p^d$, where $d = \operatorname{rank}_p \overline{G}$.
Furthermore:
	\begin{itemize}
	\item[(a)]
	If $r \not\equiv_p s$ and $v-2k+2r \not\equiv_p 0$, then $d = g+1$.
	\smallskip
	\item[(b)]
	If $r \not\equiv_p s$ and $v - 2k + 2r \equiv_p 0$, then $d = g$.
	\smallskip
	\item[(c)]
	If $r \equiv_p s$, then $d \leq \min \{ f+1, g+1 \}$.
	\smallskip
	\item[(d)]
	Let $p^m$ be the largest power of $p$ that divides $v$.
	If $p^{m+1}$ divides $v-2k+2r$, then $d \leq g$.
	\end{itemize}
\end{theorem}

Our proof of~(d) uses a strategy suggested in~\cite[\S 13.7]{BH12}.
The assumption $f \neq g$ is given only to make it obvious that $G$ has integer entries.
In fact, if $f = g$ then $v - 4k + 2\lambda + 2\mu = -1$, and there is no choice of $p$ to satisfy the hypotheses above.

\begin{proof}
Recall our notation~\eqref{eq: SRG sig eigs} for the eigenvalues of $\Sigma$.
By~\eqref{eq: SRG S1} and our assumption $v-4k+2\lambda+2\mu \equiv_p 0$, we see that $G = \Sigma - \theta_r I$ satisfies
\begin{equation}
\label{eq: 2-graph 1}
\overline{ G\Bigl( G + (\theta_r - \theta_s) I \Bigr) }
= \overline{ (v-4k+2\lambda+2\mu) J }
= 0.
\end{equation}
Therefore $\overline{G}^2 = (\theta_s - \theta_r) \overline{G}$, and it follows from Proposition~\ref{prop: orthogonal ETF from Gram matrix} that $\overline{G}$ is the Gram matrix of an ETF with the given parameters in an orthogonal geometry on $\mathbb{F}_p^d$, where $d = \operatorname{rank}_p \overline{G}$.

It remains to prove (a)--(d).
Observe that $G = \Sigma - \theta_r I$ has spectrum 
\[ \sigma(G) = {\{ (\theta_k - \theta_r)^1, 0^f, (\theta_s - \theta_r)^g \}} \]
with multiplicities denoted by superscripts, where it is possible that some eigenvalues coincide.
We have $\operatorname{rank} G = g+1$ if $\theta_k \neq \theta_r$, and $\operatorname{rank} G = g$ otherwise.

Proposition~\ref{prop: p-ranks}(a) immediately implies (a), since $\theta_k - \theta_r = v - 2k+2r$ and $\theta_s - \theta_r = 2(s-r)$.
If $\theta_k = \theta_r$, then (b) follows as well.
On the other hand, if $\theta_k \neq \theta_r$ then $G$ has minimal polynomial
\[ x(x-\theta_k + \theta_r)(x - \theta_s + \theta_r) =: (x-\theta_k + \theta_r)\, h(x), \]
and $\overline{h(G)} = 0$ by~\eqref{eq: 2-graph 1}.
In this case, (b) follows from Proposition~\ref{prop: p-ranks}(b).
For (c), we observe that $\operatorname{rank}_p \overline{G} \leq \operatorname{rank} G \leq g+1$, and if $r \equiv_p s$ then $\theta_r \equiv_p \theta_s$ and likewise
\[ \overline{G} = \overline{\Sigma - \theta_s I} \leq \operatorname{rank} (\Sigma - \theta_s I) \leq f +1. \]
Hence, $d \leq \min \{ f +1, g+1 \}$ if $r \equiv_p s$.

Finally, we prove (d).
Assume $p^{m+1}$ divides $v - 2k+2r$.
Then $p^m$ divides $k-r$.
Put $t = \tfrac{k-r}{p^m}$ and $z = \tfrac{v}{p^m}$, so that $z -2t \equiv_p 0$, and define $H = z(s-r) E_s$.
We will show $\overline{H}$ is a nonzero scalar multiple of $\overline{G}$, and it will follow that 
\[ d = \operatorname{rank}_p \overline{G} = \operatorname{rank}_p \overline{H} \leq \operatorname{rank} H = \operatorname{rank} E_s = g. \]
Using the definition of $E_s$ and the fact that $\Sigma = J - 2A - I$, we can rewrite
\[
H = zA - zrI - \tfrac{z(k-r)}{v}J = zA - zrI - tJ = (z-2t)A -t \Sigma - (zr+t) I.
\]
Since $z - 2t \equiv_p 0$ we have $\overline{H} = -t \overline{\Sigma} - (zr+t) \overline{I}$.
Next, the relation $z \equiv_p 2t$ implies $zr+t \equiv_p t(2r+1)$.
It follows that $\overline{H} = -t \overline{G}$.
Finally, $t = \frac{k-r}{p^m} \not\equiv_p 0$ since $p^{m+1}$ divides $v-2k+2r$ but not $v$.
This completes the proof.
\end{proof}

\begin{remark}
In one view of Theorem~\ref{thm: 2-graph}, the real positive-semidefinite matrix $G = \Sigma + (2r+1)I$ represents a real equiangular system that is not necessarily tight, but that becomes tight when projected into $\mathbb{F}_p^{v\times v}$.
In the case where $r \not\equiv_p s$, we could alternatively view $\overline{G}$ as arising from a real \emph{two-distance tight frame}~\cite{BGOY15} that is projected into a finite field in such a way that the two distances become opposites.
Indeed, both of $E_s$ and $I-E_r$ are Gram matrices of two-distance tight frames, and each has rational entries.
If $p$ divides $m := \frac{v-2k+2r}{\operatorname{gcd}(v,k-r)}$, then there is an integer $c$ such that $cE_s \in \mathbb{Z}^{v\times v}$ and $\overline{G} = \overline{cE_s}$.
If $p$ does not divide $m$, then we can similarly obtain $\overline{G}$ from projecting a multiple of $I-E_r$.
We omit the details.
However, in the case where $r \equiv_p s$, there are examples where $\overline{G}$ cannot be obtained by projecting any multiple of $E_s$, $E_r$, ${I-E_s}$, or ${I-E_r}$, since any integral multiple of any of these matrices is congruent to $J$ mod~$p$.
For instance, this occurs when $(v,k,\lambda,\mu) = (81,20,1,6)$ and $p=3$.

In fact, Theorem~\ref{thm: 2-graph} completely captures the phenomenon of projecting real two-distance tight frames into finite fields to make them equiangular.
Indeed, let $H \in \mathbb{R}^{n \times n}$ be the Gram matrix of a non-equiangular two-distance $c$-tight frame.
Suppose $H$ has integer entries and there is an odd prime power $p$ for which the projection $\overline{H} \in \mathbb{F}_p^{n \times n}$ is nonzero and represents an equiangular system.
By Theorem~1.2 in~\cite{BGOY15}, there is a nontrivial SRG for which $H$ is one of $cE_s$, $cE_r$, $c(I-E_r)$, or $c(I-E_s)$.
We may assume that $H = cE_s$ by replacing the graph with its complement or $H$ with $cI - H$, as necessary.
Then one can show that $f \neq g$, that $p$ divides $v - 4k + 2\lambda + 2\mu$, and that $\overline{H}$ is a scalar multiple of $\overline{\Sigma + (2r + 1) I}$.
We omit the details.
\end{remark}

\section{Examples and applications for SRGs}
\label{sec: ex}

We now give consequences of the preceding theory.
Most known SRGs imply the existence of ETFs over finite fields, and we collect small ETF sizes that occur in this way.
Using the infinite family of triangular graphs, we deduce that Gerzon's bound is attained in orthogonal geometries of infinitely many dimensions.
This stands in contrast with the real setting, where it is an open problem if Gerzon's bound is saturated in $\mathbb{R}^d$ for some $d \notin \{2,3,7,23\}$.
Many unknown SRGs would also imply the existence of ETFs, and this provides necessary conditions for their existence.
We collect this data as well.

\subsection{Existence}

\begin{example}
\label{ex: green SRGs}
Table~\ref{tbl: green SRGs} lists some ETFs in finite orthogonal geometries that arise from our constructions.
We created the table by iterating through SRG parameters in Brouwer's table~\cite{B}, as implemented in~\cite{Sage}.
For each parameter set, we proceeded as follows.
If $k \neq 2\mu$ then we applied Theorem~\ref{thm: 2-graph*} for all primes that satisfy its hypotheses.
Where $r \equiv_p s$, we created the Gram matrix for a single SRG with the given parameters, and computed its rank to find $d$.
If $k = 2\mu$ then the same construction from Theorem~\ref{thm: 2-graph*} creates a real ETF, and Proposition~\ref{prop: project real ETF} applies.
Our table omits ETFs that occur in this way, except in those cases where the resulting frame constant is $c=0$ (since then Proposition~\ref{prop: project real ETF} cannot predict the dimension).
In those cases, we again constructed the Gram matrix from a single SRG and recorded its rank.
For each parameter set, we then followed a similar procedure using Theorem~\ref{thm: 2-graph}, where the condition $k = 2\mu$ was replaced by $v = 4k - 2\lambda - 2\mu$.

Table~\ref{tbl: green SRGs} describes our results with $d < 40$.
We do not claim to record all consequences of Theorem~\ref{thm: 2-graph*} and Theorem~\ref{thm: 2-graph} from known SRGs in this range.
In particular, we did not attempt to compute dimensions for nonisomorphic SRGs with equal parameters in cases where the parameters may not predict $d$.
(In fact, it is known that $p$-rank can be an effective way to distinguish nonisomorphic SRGs~\cite{BVE92}.)

\begin{table}
\tiny{
\begin{tabular}{rrrrr}
$p$ & $d$ & $n$ & $a$ & $c$ \\ \hline
3 & 4 & 10 & 0 & 0 \\
3 & 7 & 28 & 0 & 0 \\
3 & 9 & 25 & 0 & 1 \\
3 & 9 & 37 & 0 & 1 \\
3 & 10 & 37 & 0 & 0 \\
3 & 10 & 55 & 0 & 0 \\
3 & 12 & 36 & 1 & 0 \\
3 & 12 & 49 & 0 & 1 \\
3 & 12 & 67 & 0 & 1 \\
3 & 13 & 91 & 0 & 0 \\
3 & 14 & 36 & 1 & 0 \\
3 & 14 & 45 & 1 & 0 \\
3 & 15 & 64 & 0 & 1 \\
3 & 15 & 106 & 0 & 1 \\
3 & 16 & 82 & 0 & 0 \\
3 & 16 & 136 & 0 & 0 \\
3 & 18 & 81 & 1 & 0 \\
3 & 18 & 100 & 0 & 1 \\
3 & 18 & 154 & 0 & 1 \\
3 & 19 & 49 & 1 & 1 \\
3 & 19 & 65 & 1 & 2 \\
3 & 19 & 81 & 1 & 0 \\
3 & 19 & 105 & 1 & 0 \\
3 & 19 & 190 & 0 & 0 \\
3 & 20 & 46 & 0 & 0 \\
3 & 20 & 57 & 1 & 0 \\
3 & 21 & 81 & 1 & 0 \\
3 & 21 & 121 & 0 & 1 \\
3 & 21 & 126 & 1 & 0 \\
3 & 21 & 162 & 1 & 0 \\
3 & 21 & 211 & 0 & 1 \\
3 & 22 & 50 & 1 & 2 \\
3 & 22 & 65 & 1 & 2 \\
3 & 22 & 77 & 1 & 2 \\
3 & 22 & 100 & 1 & 1 \\
3 & 22 & 145 & 0 & 0 \\
3 & 22 & 243 & 1 & 0 \\
3 & 22 & 253 & 0 & 0 \\
3 & 22 & 253 & 1 & 1 \\
3 & 22 & 276 & 1 & 0 \\
3 & 24 & 117 & 1 & 0 \\
3 & 24 & 169 & 0 & 1 \\
3 & 24 & 277 & 0 & 1 \\
3 & 25 & 101 & 1 & 2 \\
3 & 25 & 325 & 0 & 0 \\
3 & 27 & 196 & 0 & 1 \\
3 & 27 & 352 & 0 & 1 \\
3 & 28 & 100 & 1 & 1 \\
3 & 28 & 226 & 0 & 0 \\
3 & 28 & 351 & 1 & 0 \\
3 & 28 & 378 & 1 & 0 \\
3 & 28 & 406 & 0 & 0 \\
3 & 30 & 144 & 1 & 0 \\
3 & 30 & 256 & 0 & 1 \\
3 & 30 & 436 & 0 & 1 \\
3 & 31 & 122 & 1 & 2 \\
3 & 31 & 155 & 1 & 2 \\
3 & 31 & 496 & 0 & 0 \\
3 & 32 & 144 & 1 & 0 \\
3 & 32 & 177 & 1 & 0 \\
3 & 33 & 289 & 0 & 1 \\
3 & 33 & 529 & 0 & 1 \\
3 & 34 & 85 & 1 & 1 \\
3 & 34 & 325 & 0 & 0 \\
3 & 34 & 595 & 0 & 0 \\
3 & 35 & 120 & 1 & 0 \\
3 & 35 & 729 & 1 & 0 \\
3 & 35 & 1080 & 1 & 0 \\
3 & 35 & 1107 & 1 & 0 \\
3 & 36 & 361 & 0 & 1
\end{tabular}
}
\qquad
\tiny{
\begin{tabular}{rrrrr}
$p$ & $d$ & $n$ & $a$ & $c$ \\ \hline
3 & 36 & 631 & 0 & 1 \\
3 & 37 & 101 & 1 & 2 \\
3 & 37 & 112 & 1 & 1 \\
3 & 37 & 169 & 1 & 1 \\
3 & 37 & 703 & 0 & 0 \\
3 & 38 & 144 & 1 & 0 \\
3 & 39 & 225 & 1 & 0 \\
3 & 39 & 400 & 0 & 1 \\
3 & 39 & 742 & 0 & 1 \\
5 & 3 & 6 & 0 & 0 \\
5 & 9 & 26 & 0 & 0 \\
5 & 10 & 45 & 2 & 4 \\
5 & 11 & 56 & 2 & 2 \\
5 & 12 & 26 & 0 & 0 \\
5 & 12 & 78 & 2 & 3 \\
5 & 13 & 49 & 2 & 1 \\
5 & 15 & 65 & 2 & 4 \\
5 & 15 & 105 & 2 & 4 \\
5 & 16 & 81 & 2 & 2 \\
5 & 16 & 121 & 2 & 2 \\
5 & 17 & 153 & 2 & 3 \\
5 & 19 & 126 & 0 & 0 \\
5 & 20 & 56 & 0 & 2 \\
5 & 20 & 81 & 0 & 2 \\
5 & 20 & 190 & 2 & 4 \\
5 & 21 & 51 & 0 & 0 \\
5 & 21 & 176 & 0 & 0 \\
5 & 21 & 211 & 2 & 2 \\
5 & 22 & 253 & 2 & 3 \\
5 & 23 & 46 & 0 & 0 \\
5 & 23 & 101 & 0 & 0 \\
5 & 23 & 144 & 2 & 1 \\
5 & 23 & 276 & 0 & 0 \\
5 & 24 & 101 & 0 & 0 \\
5 & 25 & 81 & 0 & 2 \\
5 & 25 & 101 & 0 & 2 \\
5 & 25 & 170 & 2 & 4 \\
5 & 25 & 300 & 2 & 4 \\
5 & 26 & 196 & 2 & 2 \\
5 & 26 & 326 & 2 & 2 \\
5 & 27 & 126 & 0 & 0 \\
5 & 27 & 378 & 2 & 3 \\
5 & 30 & 121 & 0 & 2 \\
5 & 30 & 435 & 2 & 4 \\
5 & 31 & 100 & 1 & 0 \\
5 & 31 & 466 & 2 & 2 \\
5 & 32 & 176 & 0 & 0 \\
5 & 32 & 528 & 2 & 3 \\
5 & 33 & 100 & 1 & 0 \\
5 & 33 & 289 & 2 & 1 \\
5 & 35 & 126 & 0 & 0 \\
5 & 35 & 325 & 2 & 4 \\
5 & 35 & 595 & 2 & 4 \\
5 & 36 & 111 & 2 & 2 \\
5 & 36 & 361 & 2 & 2 \\
5 & 36 & 631 & 2 & 2 \\
5 & 37 & 703 & 2 & 3 \\
5 & 39 & 226 & 0 & 0 \\
7 & 12 & 66 & 3 & 6 \\
7 & 13 & 79 & 3 & 1 \\
7 & 14 & 105 & 3 & 5 \\
7 & 16 & 50 & 0 & 0 \\
7 & 17 & 81 & 3 & 4 \\
7 & 19 & 101 & 3 & 6 \\
7 & 19 & 171 & 3 & 6 \\
7 & 20 & 121 & 3 & 1 \\
7 & 20 & 191 & 3 & 1 \\
7 & 21 & 77 & 2 & 5 \\
7 & 21 & 231 & 3 & 5 \\
7 & 24 & 50 & 0 & 0
\end{tabular}
}
\qquad
\tiny{
\begin{tabular}{rrrrr}
$p$ & $d$ & $n$ & $a$ & $c$ \\ \hline
7 & 24 & 100 & 2 & 6 \\
7 & 26 & 325 & 3 & 6 \\
7 & 27 & 117 & 2 & 4 \\
7 & 27 & 352 & 3 & 1 \\
7 & 28 & 406 & 3 & 5 \\
7 & 31 & 121 & 2 & 6 \\
7 & 31 & 256 & 3 & 4 \\
7 & 33 & 290 & 3 & 6 \\
7 & 33 & 528 & 3 & 6 \\
7 & 34 & 145 & 2 & 4 \\
7 & 34 & 324 & 3 & 1 \\
7 & 34 & 562 & 3 & 1 \\
7 & 35 & 85 & 0 & 2 \\
7 & 35 & 630 & 3 & 5 \\
7 & 36 & 169 & 2 & 2 \\
7 & 37 & 344 & 0 & 0 \\
7 & 38 & 247 & 2 & 6 \\
11 & 16 & 120 & 3 & 6 \\
11 & 17 & 137 & 3 & 8 \\
11 & 18 & 171 & 3 & 1 \\
11 & 25 & 169 & 3 & 4 \\
11 & 27 & 197 & 3 & 6 \\
11 & 27 & 351 & 3 & 6 \\
11 & 28 & 225 & 3 & 8 \\
11 & 28 & 379 & 3 & 8 \\
11 & 29 & 435 & 3 & 1 \\
11 & 36 & 122 & 0 & 0 \\
11 & 37 & 223 & 5 & 1 \\
11 & 38 & 703 & 3 & 6 \\
11 & 39 & 97 & 4 & 6 \\
11 & 39 & 742 & 3 & 8 \\
13 & 7 & 14 & 0 & 0 \\
13 & 18 & 153 & 3 & 6 \\
13 & 19 & 172 & 3 & 8 \\
13 & 20 & 210 & 3 & 12 \\
13 & 29 & 225 & 3 & 4 \\
13 & 30 & 155 & 5 & 2 \\
13 & 31 & 257 & 3 & 6 \\
13 & 31 & 465 & 3 & 6 \\
13 & 32 & 289 & 3 & 8 \\
13 & 32 & 497 & 3 & 8 \\
13 & 33 & 561 & 3 & 12 \\
13 & 39 & 247 & 5 & 10 \\
17 & 9 & 18 & 0 & 0 \\
17 & 22 & 231 & 3 & 6 \\
17 & 23 & 254 & 3 & 8 \\
17 & 24 & 300 & 3 & 12 \\
17 & 36 & 127 & 5 & 3 \\
17 & 37 & 361 & 3 & 4 \\
17 & 39 & 401 & 3 & 6 \\
17 & 39 & 741 & 3 & 6 \\
19 & 24 & 276 & 3 & 6 \\
19 & 25 & 301 & 3 & 8 \\
19 & 26 & 351 & 3 & 12 \\
19 & 36 & 222 & 5 & 15 \\
23 & 28 & 378 & 3 & 6 \\
23 & 29 & 407 & 3 & 8 \\
23 & 30 & 465 & 3 & 12 \\
29 & 15 & 30 & 0 & 0 \\
29 & 34 & 561 & 3 & 6 \\
29 & 35 & 596 & 3 & 8 \\
29 & 36 & 666 & 3 & 12 \\
31 & 36 & 630 & 3 & 6 \\
31 & 37 & 667 & 3 & 8 \\
31 & 38 & 741 & 3 & 12 \\
37 & 19 & 38 & 0 & 0 \\
41 & 21 & 42 & 0 & 0 \\
53 & 27 & 54 & 0 & 0 \\
61 & 31 & 62 & 0 & 0 \\
73 & 37 & 74 & 0 & 0
\end{tabular}
}

\bigskip
\caption{
For each row above, an $(a,1,c)$-ETF of $n$ vectors exists in an orthogonal geometry on $\mathbb{F}_p^d$.
We omit data that can be deduced from the existence of a known real ETF using Proposition~\ref{prop: project real ETF}.
Example~\ref{ex: green SRGs} describes our methodology.
}
\label{tbl: green SRGs}
\end{table}
\end{example}

\begin{remark}
\label{rem: SRGs from ETFs}
Theorem~\ref{thm: Seidel Waldron} and Theorem~\ref{thm: centered or axial} describe the construction of $p$-modular SRGs from ETFs in orthogonal geometries.
In many cases, the resulting graphs are SRGs, and empirically this accounts for a large fraction of known SRGs.
By Theorem~\ref{thm: 2-graph}, a $(v,k,\lambda,\mu)$-SRG arises from the construction of Theorem~\ref{thm: centered or axial} unless $v-4k+2\lambda+2\mu$ is a power of $2$.
By Theorem~\ref{thm: 2-graph*}, it arises from the construction of Theorem~\ref{thm: Seidel Waldron} whenever there is an odd prime that divides both $k-2\mu$ and $v-3k+2\lambda+1$.
(Apply Proposition~\ref{prop: project real ETF} if $f=g$.)
Using Brouwer's table~\cite{B}, we find that $1142$ out of $1160$ known SRG parameters with $v \leq 1300$ can be constructed from an ETF using Theorem~\ref{thm: Seidel Waldron} or Theorem~\ref{thm: centered or axial}.
The exceptions are shown in Table~\ref{tbl: SRGs not from ETFs}.
In addition, $2824$ out of $2910$ feasible parameters with $v \leq 1300$ for which existence is unresolved would arise from an ETF in this way, if the SRG exists.

\begin{example}
The Steiner construction of \cite[Theorem~12.1]{S91} and \cite[Theorem~1]{FMT12} works equally well over finite fields to create real model ETFs.
It produces more ETF sizes in the finite field setting, since modular Hadamard matrices outnumber their real counterparts.
In particular, if $p$ is an odd prime and $m \equiv_p 4$, there is a $(3,1,8)$-ETF of size $\frac{m(m-1)}{2} \times m^2$ in the real model over $\mathbb{F}_p$, and it can be chosen to have centroidal symmetry.
In the Steiner construction, take the system with vertex set $[m]$ and all $2$-subsets as blocks.
A modular Hadamard matrix of the appropriate size is given by
\[
H = \left[ \begin{array}{cc}
1 & \mathbf{1}_{m-1}^\top \\
\mathbf{1}_{m-1} & 2I_{m-1} - J_{m-1}
\end{array} \right]
\in \mathbb{F}_p^{m\times m}.
\]
Combine these ingredients as in \cite[Theorem~1]{FMT12} to obtain the ETF described above.

Notably, the existence of a real ETF with size $\frac{m(m-1)}{2} \times m^2$ is currently unresolved for $m \in \{22,34\}$, but the Steiner construction provides examples in the real model over $\mathbb{F}_3$~\cite{FM:T,B}.
Furthermore, when $m$ is odd, Theorem~$A$ in~\cite{STDH07} implies there is no real ETF with size $\frac{m(m-1)}{2} \times m^2$, but the Steiner construction gives an example in the real model over some finite field as long as $|m-4|$ is not a power of~$2$.

%
%
\end{example}

\begin{table}
\begin{tabular}{rrrr}
$v$ & $k$ & $\lambda$ & $\mu$ \\ \hline
21 & 10 & 3 & 6 \\
40 & 12 & 2 & 4 \\
57 & 24 & 11 & 9
\end{tabular}
\qquad
\begin{tabular}{rrrr}
$v$ & $k$ & $\lambda$ & $\mu$ \\ \hline
70 & 27 & 12 & 9 \\
112 & 30 & 2 & 10 \\
120 & 42 & 8 & 18
\end{tabular}
\qquad
\begin{tabular}{rrrr}
$v$ & $k$ & $\lambda$ & $\mu$ \\ \hline
220 & 84 & 38 & 28 \\
280 & 117 & 44 & 52 \\
512 & 196 & 60 & 84
\end{tabular}
\bigskip
\caption{
Except for the nine parameter sets above and their complements, every known SRG with $v \leq 1300$ vertices can be constructed from an ETF over a finite field.
See Remark~\ref{rem: SRGs from ETFs}.
}
\label{tbl: SRGs not from ETFs}
\end{table}
\end{remark}

Next we interpret Theorem~\ref{thm: 2-graph} for a particular infinite family of SRGs.
Given a positive integer $d$, we write $\overline{T}(d+1)$ for the complement of the \textbf{triangular graph} whose vertices are the $2$-sets of $\{1,\dotsc,d+1\}$, with vertices adjacent in $\overline{T}(d+1)$ when they have trivial intersection.
It is an SRG with parameters $v = \binom{d+1}{2}$, $k = \tfrac{1}{2}(d-2)(d-1)$, $\lambda = \tfrac{1}{2}(d-4)(d-3)$, $\mu = \tfrac{1}{2}(d-3)(d-2)$ and $r = 1$, $s = 2-d$, $f = \tfrac{1}{2}(d^2-1)$, $g = d$ (see \cite{BVM} for the complementary parameters).

\begin{theorem}
\label{thm: triangular gerzon}
For every integer $d > 1$ and every odd prime $p$ that divides $d-7$, there exists a $(3,1,12)$-ETF of $n = \binom{d+1}{2}$ vectors in an orthogonal geometry on $\mathbb{F}_p^d$.
\end{theorem}

\begin{proof}
Take $\Gamma = \overline{T}(d+1)$ in Theorem~\ref{thm: 2-graph}.
From the above, we obtain 
\[ v - 4k + 2\lambda + 2\mu = \tfrac{1}{2}(d-7)(d-4), \]
with $v - 2k + 2r = -\tfrac{1}{2}d(d-7) \equiv_p 0$ and $r -s = d - 1 \equiv_p 6$.
Since $p$ is odd and divides $d-7$, Theorem~\ref{thm: 2-graph} produces an ETF with the given parameters in an orthogonal geometry on $\mathbb{F}_p^{d'}$.
If $p \neq 3$ then $r\not\equiv_p s$, and Theorem~\ref{thm: 2-graph}(b) gives $d' = g = d$.
If $p = 3$ then $d \equiv_p 1$, and the largest power of $p$ that divides $v \equiv_p 1$ is $p^m = 1$.
Here $p^{m+1} = p$ divides $v-2k+2r$, so we can apply Theorem~\ref{thm: 2-graph}(d) to obtain $d' \leq g = d$.
The inequality cannot be strict, or else the Gram matrix would produce an ETF that violates Gerzon's bound.
Therefore $d' = d$.
\end{proof}

\subsection{Necessary conditions}

\begin{example}
\label{ex: yellow SRGs}
Our results provide necessary conditions for the unresolved existence of SRGs whose parameters satisfy the hypotheses of Theorem~\ref{thm: 2-graph*} or Theorem~\ref{thm: 2-graph}.
Specifically, Table~\ref{tbl: yellow SRGs} lists the parameters of ETFs whose existence would be implied by such SRGs on $v \leq 300$ vertices.
(We omit the SRG parameter $\mu$ since it is uniquely determined by $v,k,\lambda$.)
Each row is a consequence of Theorem~\ref{thm: 2-graph*} (if $n = v+1$) or Theorem~\ref{thm: 2-graph} (if $n = v$).

Among the data we see that if \textbf{Conway's 99-graph} with parameters $(99,14,1,2)$ exists, then there is a $(2,1,4)$-ETF of $n=100$ vectors in an orthogonal geometry on $\mathbb{F}_5^{45}$.
The authors do not know if such an ETF exists.
It is well known that there is an ETF of $100$ vectors in $\mathbb{R}^{45}$, but if we attempt to project it into $\mathbb{F}_5$ with Proposition~\ref{prop: project real ETF}, then we obtain a $(4,1,0)$-ETF in dimension $d' \leq 45$.
These parameters do not match the ones predicted for Conway's graph, and in any case we cannot be sure from Proposition~\ref{prop: project real ETF} that $d'=45$.

\begin{table}
\tiny{
\begin{tabular}{rrr|rrrrr}
$v$ & $k$ & $\lambda$ & $p$ & $d$ & $n$ & $a$ & $c$ \\ \hline
69 & 48 & 32 & 13 & 24 & 69 & 5 & 3 \\
85 & 54 & 33 & 7 & 35 & 85 & 0 & 4 \\
85 & 70 & 57 & 5 & 35 & 86 & 0 & 4 \\
85 & 70 & 57 & 13 & 35 & 85 & 5 & 1 \\
88 & 27 & 6 & 5 & 32 & 88 & 2 & 3 \\
99 & 14 & 1 & 5 & 45 & 100 & 2 & 4 \\
100 & 33 & 8 & 3 & 34 & 101 & 1 & 2 \\
111 & 30 & 5 & 19 & 36 & 111 & 7 & 1 \\
112 & 36 & 10 & 3 & 48 & 112 & 0 & 2 \\
115 & 18 & 1 & 3 & 46 & 115 & 1 & 1 \\
115 & 18 & 1 & 17 & 45 & 115 & 7 & 16 \\
120 & 34 & 8 & 7 & 52 & 121 & 2 & 6 \\
120 & 84 & 58 & 3 & 57 & 121 & 0 & 2 \\
121 & 36 & 7 & 3 & 37 & 121 & 1 & 1 \\
121 & 36 & 7 & 5 & 36 & 121 & 2 & 2 \\
121 & 48 & 17 & 3 & 48 & 121 & 0 & 1 \\
133 & 32 & 6 & 3 & 57 & 133 & 0 & 2 \\
133 & 32 & 6 & 11 & 56 & 133 & 9 & 9 \\
133 & 88 & 57 & 3 & 57 & 133 & 0 & 1 \\
133 & 88 & 57 & 5 & 56 & 133 & 4 & 2 \\
133 & 108 & 87 & 5 & 57 & 133 & 2 & 3 \\
133 & 108 & 87 & 11 & 56 & 133 & 7 & 7 \\
136 & 105 & 80 & 3 & 52 & 137 & 1 & 2 \\
136 & 105 & 80 & 11 & 52 & 136 & 7 & 9 \\
162 & 21 & 0 & 5 & 57 & 163 & 2 & 3 \\
162 & 21 & 0 & 7 & 56 & 162 & 0 & 4 \\
162 & 92 & 46 & 7 & 24 & 163 & 5 & 1 \\
162 & 112 & 76 & 13 & 64 & 162 & 9 & 11 \\
162 & 138 & 117 & 7 & 70 & 162 & 0 & 4 \\
162 & 138 & 117 & 17 & 70 & 163 & 7 & 1 \\
169 & 42 & 5 & 3 & 43 & 170 & 1 & 2 \\
169 & 42 & 5 & 5 & 43 & 169 & 2 & 1 \\
169 & 42 & 5 & 7 & 42 & 169 & 0 & 5 \\
169 & 56 & 15 & 3 & 57 & 169 & 0 & 2 \\
169 & 56 & 15 & 5 & 56 & 169 & 4 & 1 \\
169 & 70 & 27 & 3 & 70 & 169 & 2 & 2 \\
175 & 108 & 63 & 13 & 43 & 175 & 7 & 4 \\
176 & 25 & 0 & 3 & 55 & 176 & 1 & 2 \\
176 & 25 & 0 & 7 & 56 & 176 & 0 & 6 \\
176 & 25 & 0 & 17 & 56 & 177 & 7 & 3 \\
176 & 70 & 24 & 5 & 56 & 177 & 4 & 3 \\
183 & 52 & 11 & 5 & 61 & 184 & 4 & 1 \\
183 & 52 & 11 & 29 & 60 & 183 & 9 & 26 \\
189 & 128 & 82 & 11 & 29 & 189 & 5 & 3 \\
189 & 140 & 103 & 5 & 91 & 189 & 1 & 4 \\
190 & 144 & 108 & 5 & 76 & 191 & 4 & 4 \\
196 & 39 & 2 & 3 & 49 & 197 & 1 & 2 \\
196 & 39 & 2 & 7 & 49 & 197 & 0 & 5 \\
196 & 39 & 2 & 31 & 48 & 196 & 7 & 26 \\
196 & 45 & 4 & 3 & 46 & 196 & 1 & 1 \\
196 & 75 & 26 & 3 & 76 & 197 & 2 & 1 \\
196 & 114 & 59 & 5 & 25 & 196 & 0 & 2 \\
204 & 28 & 2 & 13 & 84 & 204 & 9 & 7 \\
204 & 140 & 94 & 3 & 69 & 205 & 0 & 1 \\
205 & 68 & 15 & 3 & 40 & 205 & 1 & 1 \\
205 & 68 & 15 & 5 & 40 & 205 & 2 & 4 \\
208 & 45 & 8 & 5 & 91 & 209 & 1 & 4 \\
209 & 156 & 115 & 5 & 76 & 209 & 4 & 1 \\
209 & 156 & 115 & 7 & 77 & 210 & 2 & 5 \\
209 & 156 & 115 & 11 & 77 & 209 & 9 & 4 \\
210 & 33 & 0 & 5 & 55 & 210 & 2 & 4 \\
210 & 33 & 0 & 7 & 56 & 211 & 0 & 3 \\
210 & 76 & 26 & 5 & 96 & 211 & 3 & 3 \\
210 & 132 & 82 & 3 & 100 & 211 & 1 & 1 \\
216 & 129 & 72 & 13 & 44 & 217 & 7 & 10 \\
216 & 172 & 136 & 5 & 86 & 216 & 4 & 4 \\
217 & 128 & 72 & 3 & 63 & 217 & 0 & 2 \\
220 & 72 & 22 & 3 & 100 & 220 & 1 & 1 \\
225 & 48 & 3 & 7 & 49 & 225 & 0 & 2 \\
\end{tabular}
}
\qquad
\tiny{
\begin{tabular}{rrr|rrrrr}
$v$ & $k$ & $\lambda$ & $p$ & $d$ & $n$ & $a$ & $c$ \\ \hline
225 & 64 & 13 & 7 & 64 & 225 & 2 & 2 \\
225 & 128 & 64 & 5 & 25 & 226 & 0 & 3 \\
231 & 160 & 110 & 5 & 111 & 231 & 3 & 3 \\
232 & 33 & 2 & 3 & 87 & 232 & 0 & 1 \\
232 & 33 & 2 & 19 & 88 & 232 & 9 & 3 \\
232 & 33 & 2 & 23 & 88 & 233 & 9 & 22 \\
232 & 63 & 14 & 3 & 88 & 233 & 2 & 1 \\
232 & 63 & 14 & 11 & 88 & 232 & 0 & 6 \\
232 & 150 & 95 & 11 & 88 & 232 & 0 & 8 \\
232 & 154 & 96 & 3 & 28 & 232 & 2 & 2 \\
232 & 154 & 96 & 37 & 29 & 233 & 5 & 7 \\
235 & 52 & 9 & 3 & 94 & 235 & 2 & 2 \\
235 & 52 & 9 & 7 & 95 & 236 & 4 & 5 \\
235 & 52 & 9 & 23 & 95 & 235 & 11 & 3 \\
235 & 192 & 156 & 11 & 95 & 235 & 9 & 2 \\
236 & 180 & 135 & 37 & 60 & 236 & 7 & 30 \\
238 & 75 & 20 & 7 & 84 & 238 & 4 & 2 \\
243 & 66 & 9 & 13 & 44 & 243 & 7 & 10 \\
245 & 52 & 3 & 3 & 49 & 245 & 1 & 2 \\
245 & 52 & 3 & 23 & 50 & 245 & 7 & 9 \\
245 & 180 & 131 & 3 & 100 & 245 & 2 & 1 \\
245 & 180 & 131 & 19 & 101 & 245 & 11 & 9 \\
246 & 85 & 20 & 7 & 42 & 246 & 0 & 5 \\
249 & 88 & 27 & 17 & 84 & 249 & 11 & 15 \\
250 & 81 & 24 & 7 & 106 & 250 & 6 & 2 \\
250 & 153 & 88 & 3 & 46 & 250 & 1 & 1 \\
259 & 42 & 5 & 5 & 111 & 259 & 1 & 4 \\
259 & 42 & 5 & 23 & 112 & 259 & 11 & 1 \\
261 & 64 & 14 & 5 & 116 & 261 & 3 & 3 \\
261 & 64 & 14 & 13 & 117 & 261 & 0 & 2 \\
261 & 176 & 112 & 5 & 30 & 261 & 0 & 3 \\
261 & 180 & 123 & 13 & 117 & 261 & 0 & 4 \\
261 & 208 & 165 & 5 & 116 & 261 & 1 & 1 \\
261 & 208 & 165 & 19 & 117 & 261 & 11 & 7 \\
265 & 96 & 32 & 3 & 106 & 266 & 1 & 2 \\
265 & 96 & 32 & 17 & 105 & 265 & 13 & 15 \\
266 & 45 & 0 & 13 & 56 & 266 & 7 & 4 \\
273 & 80 & 19 & 3 & 91 & 274 & 2 & 2 \\
273 & 80 & 19 & 41 & 90 & 273 & 11 & 32 \\
273 & 200 & 145 & 7 & 105 & 273 & 4 & 2 \\
276 & 75 & 10 & 3 & 46 & 277 & 1 & 1 \\
276 & 75 & 10 & 11 & 45 & 276 & 7 & 7 \\
276 & 75 & 18 & 11 & 116 & 277 & 2 & 8 \\
276 & 165 & 92 & 17 & 46 & 277 & 7 & 10 \\
280 & 62 & 12 & 3 & 124 & 280 & 1 & 1 \\
280 & 62 & 12 & 17 & 125 & 281 & 13 & 11 \\
280 & 216 & 166 & 3 & 136 & 280 & 1 & 1 \\
280 & 216 & 166 & 5 & 136 & 281 & 3 & 3 \\
285 & 64 & 8 & 7 & 76 & 285 & 2 & 4 \\
285 & 64 & 8 & 11 & 75 & 285 & 9 & 10 \\
286 & 95 & 24 & 3 & 66 & 286 & 0 & 2 \\
286 & 95 & 24 & 5 & 66 & 287 & 4 & 3 \\
286 & 160 & 84 & 3 & 66 & 286 & 0 & 1 \\
288 & 41 & 4 & 29 & 124 & 289 & 11 & 24 \\
288 & 182 & 106 & 3 & 28 & 289 & 2 & 2 \\
288 & 245 & 208 & 5 & 141 & 289 & 1 & 4 \\
289 & 72 & 11 & 3 & 72 & 289 & 0 & 1 \\
289 & 72 & 11 & 7 & 73 & 289 & 2 & 6 \\
289 & 90 & 23 & 3 & 91 & 290 & 2 & 1 \\
289 & 90 & 23 & 5 & 91 & 289 & 1 & 4 \\
289 & 90 & 23 & 7 & 90 & 289 & 4 & 6 \\
289 & 108 & 37 & 3 & 109 & 289 & 1 & 1 \\
289 & 108 & 37 & 5 & 108 & 289 & 3 & 4 \\
289 & 126 & 53 & 3 & 126 & 289 & 0 & 1 \\
297 & 104 & 31 & 7 & 89 & 297 & 4 & 1 \\
300 & 184 & 108 & 7 & 70 & 301 & 2 & 5 \\
300 & 230 & 175 & 13 & 116 & 301 & 11 & 4 \\
300 & 273 & 248 & 11 & 118 & 301 & 7 & 9 \\
300 & 273 & 248 & 13 & 118 & 300 & 7 & 7
\end{tabular}
}

\bigskip
\caption{
Each row contains parameters of an SRG whose existence is unresolved.
If such an SRG exists, then there is an $(a,1,c)$-ETF of $n$ vectors in an orthogonal geometry on $\mathbb{F}_p^d$.
See Example~\ref{ex: yellow SRGs} for details.
}
\label{tbl: yellow SRGs}
\end{table}
\end{example}

\begin{example}
The \textbf{missing Moore graph} refers to a possible SRG with parameters $(3250, 57, 0, 1)$~\cite{HS60}.
At the time of writing its existence is unresolved.
If such an SRG exists, then it produces a $(4,1,8)$-ETF of $n=3251$ vectors in an orthogonal geometry on $\mathbb{F}_{11}^{1521}$ via Theorem~\ref{thm: 2-graph*}.
\end{example}

\begin{example}
It is known that there is no $(841,200,87,35)$-SRG~\cite{Neu81}.
Theorem~\ref{thm: 2-graph} gives another proof of this fact, since the complement of any such SRG would produce a $(0,1,1)$-ETF of $n=841$ vectors in an orthogonal geometry on $\mathbb{F}_5^{40}$, and this violates Gerzon's bound.
\end{example}

\begin{example}
\label{ex: Gerzon}
Table~\ref{tbl: Gerzon} summarizes feasible SRG parameters with $v \leq 1300$ that imply equality in Gerzon's bound via Theorem~\ref{thm: 2-graph}.
For the first and last rows, Theorem~\ref{thm: 2-graph} gives only an upper bound on $d$, but as in the proof of Theorem~\ref{thm: triangular gerzon} the inequality cannot be strict.

The authors do not know how to construct an ETF with any of the parameters given in Table~\ref{tbl: Gerzon}.
In particular, every instance of Gerzon equality that we currently know satisfies the identity $a^2 = (d+2)b$, but this equation does not occur anywhere in Table~\ref{tbl: Gerzon} (where $b=1$).
Proposition~\ref{prop: Welch} implies that $a^2 = (d+2)b$ whenever Gerzon's bound is saturated, $d\not\equiv_p 0$, and $d\not\equiv_p 1$.
The authors do not know if these conditions can be removed.

\begin{table}
\begin{tabular}{rrrr|rrrr}
$v$ & $k$ & $\lambda$ & $\mu$ & $p$ & $d$ & $a$ & $c$ \\ \hline
351 & 210 & 113 & 144 & 5 & 26 & 0 & 0 \\
703 & 520 & 372 & 420 & 3 & 37 & 2 & 2 \\
1081 & 486 & 177 & 252 & 5 & 46 & 2 & 2 \\
1275 & 728 & 379 & 464 & 7 & 50 & 0 & 0
\end{tabular}
\medskip
\caption{
If an SRG with the given parameters exists, then there is an $(a,1,c)$-ETF consisting of $n=\binom{d+1}{2}$ vectors in an orthogonal geometry on $\mathbb{F}_p^d$.
See Example~\ref{ex: Gerzon}.}
\label{tbl: Gerzon}
\end{table}
\end{example}

\begin{remark}
Since ETFs over finite fields give necessary conditions for the existence of SRGs, it would be helpful to have more machinery to prove nonexistence of equiangular systems with given parameters.
Gerzon's bound on the size of an equiangular system gives one such tool, but it is a blunt instrument since it takes little account of the parameters $(a,b)$ of an equiangular system (other than to require $a^2 \neq b$).
In comparison with the real case, it is natural to ask if there may be ``relative bounds'' on the number of vectors in an $(a,b)$-equiangular system that take greater advantage of the known parameters.
We leave this as an open problem.
\end{remark}

\begin{problem}[\cite{FFF1}]
Given an orthogonal geometry over $\mathbb{F}_q$ and parameters $a,b \in \mathbb{F}_q$, find an upper bound on the size of an $(a,b)$-equiangular system that improves upon Gerzon's bound wherever possible.
\end{problem}

\section{Failure of Gerzon equality in dimension~5}
\label{sec: 5x15}

By leveraging a finite search space, we now show that Gerzon's bound cannot be attained in any finite orthogonal geometry of dimension~5 over a field of odd characteristic.
This is the smallest dimension for which Theorem~\ref{thm: triangular gerzon} fails to provide an instance of Gerzon equality.

\begin{theorem}[Computer-assisted result]
\label{thm: no Gerzon in dim 5}
For any choice of odd prime power $q$, there does not exist an ETF of $n=15$ vectors in any orthogonal geometry on $\mathbb{F}_q^5$.
In particular, for any choice of $a,b \in \mathbb{F}_q$ with $b \neq a^2$, there is no $(a,b)$-equiangular system of $n=15$ vectors in any orthogonal geometry on $\mathbb{F}_q^5$.
\end{theorem}

Our computer-assisted proof implements the following lemma, which can also be generalized for unitary geometries.
Similar methods have been used to prove nonexistence of SRGs~\cite{AM18,AM20}.
Recall that a \textbf{clique} in a graph refers to a complete subgraph.
A clique is \textbf{maximal} if it is not properly contained in any other clique.
The \textbf{clique number} of a graph is the size of its largest clique.

\begin{lemma}
\label{lem: clique search}
Suppose there is an $(a,1)$-equiangular system $\Phi$ of $n$ distinct vectors in an orthogonal geometry on $\mathbb{F}_q^d$, such that $\operatorname{rank} \Phi > 1$.
Choose any two linearly independent vectors $\psi_1,\psi_2 \in \mathbb{F}_q^d$ satisfying $\langle \psi_1, \psi_1 \rangle = \langle \psi_2, \psi_2 \rangle = a$ and $\langle \psi_1, \psi_2 \rangle = 1$.
Define
\[
V = \{ \psi \in \mathbb{F}_q^d : \langle \psi, \psi \rangle = a, \quad \langle \psi, \psi_1 \rangle = 1, \quad \langle \psi, \psi_2 \rangle^2 = 1 \},
\]
and let $\Gamma$ be the \textbf{compatibility graph} on vertex set $V$ with distinct vertices $\varphi, \psi$ adjacent whenever ${\langle \varphi, \psi \rangle^2 = 1}$.
Then $\Gamma$ contains a clique $K \subseteq V$ of size $n-2$, and $\Psi = \{ \psi_1, \psi_2 \} \cup K$ is an $(a,1)$-equiangular system of $n$ vectors.
Furthermore, if $\Phi$ is an $(a,1,c)$-ETF, then $K$ can be chosen so that $\Psi$ is an ETF with the same parameters.
\end{lemma}

\begin{proof}
Denote $\Phi = \{ \varphi_j \}_{j \in [n]}$, where $\varphi_1$ and $\varphi_2$ are linearly independent.
For any choice of signs $\epsilon_j \in \{\pm 1\}$, the system $\Phi' = \{ \epsilon_j \varphi_j \}_{j\in [n]}$ is again an $(a,1)$-equiangular system, and if $\Phi$ is an $(a,1,c)$-ETF then so is $\Phi'$.
By choosing the signs appropriately, we may therefore assume that $\langle \varphi_1, \varphi_j \rangle = 1$ for every $j \geq 2$.
The linear transformation $T \colon \operatorname{span} \{ \varphi_1, \varphi_2 \} \to \operatorname{span} \{ \psi_1, \psi_2 \}$ given by $T(\varphi_j) = \psi_j$ is easily seen to be an isometric isomorphism.
By Witt's Extension Theorem (Theorem~5.2 of~\cite{Gr02}) it extends to an isometric isomorphism $U \colon \mathbb{F}_q^d \to \mathbb{F}_q^d$.
For every $j \geq 3$, $\psi_j := U(\varphi_j)$ belongs to $V$, and $K := \{ \psi_j : j \geq 3\}$ is a clique in $\Gamma$.
The proof is complete since $\Psi = \{ \psi_1,\psi_2\} \cup K = \{ \psi_j \}_{j\in [n]}$ is the image of $\Phi$ under an isometric isomorphism.
\end{proof}

Our code for the following computer-assisted proof is provided as an ancillary file with the arXiv version of this paper.

\begin{proof}[Computer-assisted proof of Theorem~\ref{thm: no Gerzon in dim 5}]
Suppose for the sake of contradiction that there is an $(a,b,c)$-ETF of $n=15$ vectors in an orthogonal geometry on $\mathbb{F}_q^5$.
As explained in Remark~\ref{rem: rescale Gram}, we may assume that $b=1$, and, by Theorem~\ref{thm: orth ETF in Fp}, we may take $q=p$ to be prime with $a,c \in \mathbb{F}_p$.
Furthermore, Proposition~\ref{prop: no wrap implies real} implies that $p \leq 2n-5 = 25$, since there is no ETF of $n=15$ vectors in $\mathbb{R}^5$ by~\cite[Corollary~16]{STDH07}.
If $p \neq 5$ then we may divide~\eqref{eq: nadc} to obtain $c = 3a$, and then Proposition~\ref{prop: Welch} states $2a^2 = 14$, or $a^2 = 7$.
Since $7$ is a square, we conclude that $p \in \{ 3, 5, 7, 19 \}$.

When $p \in \{3, 7, 19\}$, there are at most two possibilities for $a \in \mathbb{F}_p$ and two possible orthogonal geometries on $\mathbb{F}_p^5$.
For each possibility, we used GAP~\cite{gap} to implement a choice of $\psi_1,\psi_2 \in \mathbb{F}_p^5$ as in Lemma~\ref{lem: clique search}.
Using the GRAPE package for GAP~\cite{grape}, we then implemented the compatibility graph $\Gamma$ of Lemma~\ref{lem: clique search} and computed its clique number.
The results are summarized in Table~\ref{tbl: no 5 x 15 data}.
In each case, we found the clique number was less than $n-2 = 13$.
By Lemma~\ref{lem: clique search}, we conclude that there is no $(a,1)$-equiangular system of $n=15$ vectors in an orthogonal geometry on $\mathbb{F}_p^5$ for $p \in \{3,7,19\}$.

\begin{table}
\begin{tabular}{rrrrr}
$p$ & $s$ & $a$ & $| V |$ & $\omega(\Gamma)$ \\ \hline
3&0&1&24&9\\ 
3&0&2&21&5\\ 
3&1&1&12&5\\ 
3&1&2&33&9\\ 
5&0&0&40&4\\ 
5&0&1&70&25\\ 
5&0&2&55&6\\ 
\end{tabular}
\qquad
\begin{tabular}{rrrrr}
$p$ & $s$ & $a$ & $| V |$ & $\omega(\Gamma)$ \\ \hline
5&0&3&55&6\\ 
5&0&4&145&25\\ 
5&1&0&60&6\\ 
5&1&1&30&5\\ 
5&1&2&45&8\\ 
5&1&3&45&8\\ 
5&1&4&105&5\\
\end{tabular}
\qquad
\begin{tabular}{rrrrr}
$p$ & $s$ & $a$ & $| V |$ & $\omega(\Gamma)$ \\ \hline 
7&0&0&98&4\\ 
7&1&0&98&4\\ 
19&0&8&722&3\\ 
19&0&11&722&3\\ 
19&1&8&722&3\\ 
19&1&11&722&3\\
~
\end{tabular}
\medskip
\caption{
Data for the computer-assisted proof of Theorem~\ref{thm: no Gerzon in dim 5}.
The second column describes the type of geometry on $\mathbb{F}_p^5$, where $s=0$ for square-type discriminant and $s=1$ otherwise.
For $a$ as in the third column we implemented the graph $\Gamma$ of Lemma~\ref{lem: clique search} with the GRAPE package for GAP~\cite{grape,gap}.
The number of vertices appears in the fourth column, and the last column gives the clique number as found by GRAPE.
}
\label{tbl: no 5 x 15 data}
\end{table}

It remains to consider $p=5$.
Turning to GAP again, we proceeded as above, testing all possible values of $a \in \mathbb{F}_5$.
Here again, we found the clique number was less than ${n-2}=13$ in all but two cases: for the geometry with square-type discriminant and $a \in \{1,4\}$, the clique number was $25$.
Hence, we are in one of these cases.
Taking $V$, $\Gamma$, and $K$ as in Lemma~\ref{lem: clique search}, there must be a maximal clique $M$ that contains $K$, and then $\Psi' := \{ \psi_1, \psi_2 \} \cup M$ contains all the vectors of the frame $\Psi = \{ \psi_1,\psi_2 \} \cup K$, so that $\operatorname{rank} \Psi' = 5$.
However, for both of the remaining cases we checked every maximal clique $M$ in $\Gamma$ of size $13 \leq |M| \leq 25$, and in each case we found $\operatorname{rank} \Psi' \in \{3,4\}$.
(There were 16 such maximal cliques for $a=1$, and 156 for $a=4$.)
This is a contradiction.
Therefore no such ETF exists.
By Proposition~\ref{prop: Gerzon}, neither does there exist an $(a,b)$-equiangular system of $15$ vectors with $a^2 \neq b$ in any orthogonal geometry of dimension~5 over a finite field of odd characteristic.
\end{proof}

\section*{Acknowledgments}
GRWG was partially supported by the Singapore Ministry of Education Academic Research Fund (Tier 1); grant numbers: RG29/18 and RG21/20.
JJ was supported by NSF DMS 1830066.
DGM was partially supported by AFOSR FA9550-18-1-0107 and NSF DMS 1829955.
This project began at the 2018 MFO Mini-Workshop on Algebraic, Geometric, and Combinatorial Methods in Frame Theory.
The authors thank the other participants, Matt Fickus, and Steve Flammia for comments that motivated the authors or provided insight.

\bigskip

\bibliographystyle{abbrv}
\bibliography{refsff2}

\end{document}